\documentclass{article}

\usepackage{cleveref}
\usepackage{lipsum}
\usepackage{amsfonts}
\usepackage{amsmath}
\usepackage{amsthm}
\usepackage{placeins}
\usepackage{enumerate}

\usepackage{graphicx}
\usepackage{epstopdf}
\usepackage{algorithmic}
\usepackage{mathrsfs}

\usepackage[bottom=1in, top=1in, left=1in, right=1in]{geometry}

\newtheorem{thm}{Theorem}[section]
\numberwithin{equation}{section}
\numberwithin{table}{section}
\numberwithin{figure}{section}
\newtheorem{lem}[thm]{Lemma}
\newtheorem{prop}[thm]{Proposition}
\newtheorem{defn}[thm]{Definition}
\newtheorem{rmk}[thm]{Remark}
\newtheorem{crl}[thm]{Corollary}

\DeclareMathOperator{\Span}{span}

\title{Convergence of Finite Element Methods for Singular Stochastic Control}
\author{Martin G. Vieten\thanks{Department of Mathematical Sciences, University of Wisconsin-Milwaukee (\email{mgvieten@uwm.edu}).}
\and Richard H. Stockbridge \thanks{Department of Mathematical Sciences, University of Wisconsin-Milwaukee (\email{stockbri@uwm.edu}).
This research was supported in part by the Simons Foundation under grant award 246271.}}

\begin{document}
\begin{center}
\LARGE \textbf{Convergence of Finite Element Methods for\\ Singular Stochastic Control\\[0.3cm]}
\normalsize Martin G. Vieten (University of Wisconsin-Milwaukee, mgvieten@uwm.edu)\\
Richard H. Stockbridge (University of Wisconsin-Milwaukee)
\end{center}

\noindent \textbf{Abstract:} A numerical method is proposed for a class of stochastic control problems including singular behavior. This method solves an 
infinite-dimensional linear program equivalent to the stochastic control problem using a finite element type approximation, which results
in a solvable finite-dimensional program. The discretization scheme as well as the necessary assumptions are discussed, and a detailed 
convergence analysis for the discretization scheme is given. Its performance is illustrated by two examples featuring a long-term average cost criterion.\\[0.1cm]

\normalsize

\section{Introduction}
\subsection{Motivation and Literature}
This paper considers singular stochastic control problems for a process $X$ whose dynamics are initially specified by a stochastic differential
equation (SDE)
\begin{equation}
\label{introduction:sde}
 dX_t = b(X_t, u_t)dt + \sigma(X_t, u_t)dW_t +d\xi_t,\quad X_0 = x_0,
\end{equation}
where $W$ is a Brownian motion process and $\xi$ is another stochastic process that evolves singularly in time. The process $u$ represents the
control influencing the evolution of $X$. Given 
two cost functions $\tilde{c}_0$ and $\tilde{c}_1$, $u$ has to be chosen from a set of admissible controls in such a way that it minimizes
either a long-term average cost criterion
\begin{equation}
\label{introduction:lta-criterion}
\limsup_{t\rightarrow \infty} \frac{1}{t}\mathbb{E}\left[\int_0^t \tilde{c}_0(X_s,u_s)\,ds
 + \int_0^t \tilde{c}_1(X_s,u_s)\,d\xi_s \right]
 \end{equation}
 or a discounted infinite horizon cost criterion
 \begin{equation}
 \label{introduction:ihd-criterion}
 \mathbb{E}\left[\int_0^\infty e^{-\alpha s} \tilde{c}_0(X_s,u_s)\,ds
 + \int_0^\infty e^{-\alpha s} \tilde{c}_1(X_s,u_s)\,d\xi_s \right],
 \end{equation}
 for some discounting rate $\alpha>0$. Such control problems are considered in a relaxed sense by using a martingale problem formulation involving the infinitesimal generators of $X$,
 and an equivalent infinite-dimensional linear program for the expected occupation measures of both the process $X$ and the control $u$. 
 Approximate solutions to this linear program are attained by discretizing the infinite-dimensional constraint space of functions using a 
 finite element approach, and discrete approximations of the expected occupation measures. The $\epsilon$-optimality of approximate
 solutions is shown and the method is applied to two example problems.\\
  The classical analytic approach to stochastic control problems is given by methods based on the dynamic programming principle, as presented
 in Fleming and Rishel \cite{fleming-opt-control} or Fleming and Soner \cite{fleming-viscosity}. Central to these methods is the solution
 of the so-called Hamilton-Jacobi-Bellman (HJB) equation. Numerical methods can be derived by solving
 a control problem for an approximate, discrete Markov chain, as extensively discussed in Kushner and Dupuis \cite{kushner-depuis}, or
 by using discrete methods to approximate the solution to the HJB equation, frequently considering viscosity solutions. 
 An example is given in Kumar and
 Muthuraman \cite{kumar-kumar}. Another numerical
technique using dynamic programming was analyzed in Anselmi et. al. \cite{dufour-continuous-time}.\\
 As an alternative, linear programming approaches have been instrumental in the analytic treatment of various stochastic control problems.
 The first example is given in Manne \cite{manne-lp}, where an ergodic Markov chain for an inventory problem under long-term average costs is
 analyzed. Bhatt and Borkar \cite{bhatt-borkar} as well as Kurtz and Stockbridge \cite{kurtz-stockbridge-existence} investigated the linear 
 programming approach for solutions of controlled martingale problems using long-term average and discounted cost criteria for infinite horizon problems, as well as
 finite horizon and first exit problems for absolutely continuous control. 
 Taksar \cite{taksar-linear-programming} establishes equivalence between a linear program and a stochastic control problem for a
 multi-dimensional diffusion with singular control. Jump diffusions of Levy-Ito type are considered by Serrano \cite{serrano-levy-ito}.\\
 To provide an alternative to numerical techniques based on the dynamic programming principle, the linear programming approach has been
 exploited using various discretization techniques. A very general setting can be found in Mendiondo and Stockbridge \cite{stockbridge-mendiondo}. Moment-based approaches have
been used in a line of publications, as can be seen in Helmes et. al. \cite{helmes-roehl-stockbridge-moments} and 
Lasserre and Prieto-Rumeau \cite{lasserre-moments}.
Recent research by Kaczmarek et. al. \cite{kaczmarek} and Rus \cite{rus-thesis} has been investigating a novel approximation technique for
the linear programming formulation by borrowing ideas from the finite element method used for solving partial differential equations.
A discretization of the occupation measures (by discretizing their densities) and the linear constraints with a finite set of basis functions
gives a solvable finite-dimensional linear program.
Kaczmarek et. al. \cite{kaczmarek} indicated that a finite element discretization approach may outperform Markov chain approximation
methods as well as a finite difference approximation to the Hamilton-Jacobi-Bellman equation stemming from the dynamic programming approach.
However, no analytic treatment of the convergence properties was provided.\\
The present paper closes this gap by providing a modified finite element based approximation scheme for which convergence of
the computed cost criterion values can be guaranteed. To this end, the approximation scheme is split up in several steps which either
deal with the discretization of the measures or the constraints. The separate steps are set up in such a way that convergence of the 
discrete optimal solutions to the analytic optimal solution can be proven. The proofs are, on one hand, based on the concept of weak
convergence of measures, and on a detailed analysis of discretized approximations of the measures on the other hand.\\
The paper is structured as follows. The next subsection presents the notation 
and formally introduces the linear programming formulation for singular stochastic control problems, along with a review of important results
from the literature. The approximation scheme is discussed in Section \ref{approximation-section}. Then, we provide the convergence proof for
this scheme in Section \ref{convergence} and illustrate the performance of the numerical method on two examples in Section \ref{example-section}. A short outlook
on possible research directions concludes this paper. Additional proofs needed to prove the results 
from Section \ref{convergence} are given in Appendix \ref{appendix_a}.
 
\subsection{Notation and Formalism}
\label{introduction:notation-and-formalism}
The natural numbers are denoted by $\mathbb{N}$, and the non-negative integers are $\mathbb{N}_0 = \mathbb{N}\cup \{0\}$. The symbol used for the
real numbers is $\mathbb{R}$, and that for the non-negative real numbers is $\mathbb{R}_+$. The space of $n$-dimensional vectors is
$\mathbb{R}^n$, and the space of $n$ by $m$ matrices is
$\mathbb{R}^{n\times m}$.\\
The set of continuous functions on a topological space $S$ is denoted by $C(S)$. The set of twice differentiable functions on $S$ is denoted by $C^2(S)$,
while its subset of twice differentiable functions with compact support is referred to by $C^2_c(S)$. 
The space of uniformly continuous, bounded functions is denoted by $C^u_b(S)$.
On a function space,
$\|\cdot\|_\infty$ refers to the uniform norm of functions. On $\mathbb{R}^n$, $\|\cdot\|_\infty$ refers to the maximum norm of vector components,
while on $\mathbb{R}^{n\times m}$, it refers to the maximum absolute row sum norm. The space of Lebesgue integrable functions
is $L^1(S)$. For any given function $f$, let $f^+:E\ni x \mapsto f^+(x) := \max(f(x),0)$ be the 
positive part of a function $f$. \\
In terms of measurable spaces, we use $\mathscr{B}(S)$ to describe the $\sigma$-algebra of Borel sets on a topological space $S$.
Given a measurable space $(\Omega, \mathscr{F})$, the set of probability measures on $\Omega$ is $\mathcal{P}(\Omega)$, while the set of finite Borel measures
is denoted by $\mathcal{M}(\Omega)$. The symbol $\delta_{\{s\}}$ denotes the Dirac measure on $s\in S$. When using the differential $dx$
as an integrator, it is
understood that this refers to integration by Lebesgue measure. When we explicitly refer to the Lebesgue measure, we use the symbol $\lambda$.
A Brownian motion process is denoted by the symbol $W$.\\

Consider the SDE given by (\ref{introduction:sde}). We assume that $X_t \in E=[e_l, e_r]$, with $\infty<e_l<e_r<\infty$,  and 
$u_t\in U=[u_l, u_r]$, with $\infty<u_l<u_r<\infty$, for all $t\geq 0$.
$E$ and $U$ are called the state space and control space, respectively. The coefficient functions $b:E\times U\mapsto \mathbb{R}$ and
$\sigma: E\times U \mapsto \mathbb{R}_+-\{0\}$ are called the drift and diffusion functions. They are assumed to be continuous. 
The process $\xi$ is a singular stochastic process stemming from the behavior
of $X$ at the boundaries of the state space $e_l$ and $e_r$, and is given by either a reflection, a jump or a combination
of both. The infinitesimal generators of a process solving (ref{introduction:sde}) are $\tilde{A}: C_c^2(E) \mapsto C(E\times U)$, called
the continuous generator, and $B: C_c^2(E) \mapsto C(E\times U)$, called the singular generator. 
For $f\in C_c^2(E)$, $\tilde{A}$ is defined by $\tilde{A}f(x,u)=b(x,u)f'(x)+\sigma^2(x,u)f''(x)$. $B$ is defined by either of
\begin{equation}
 \label{introduction:notation-and-formalism:form-of-b}
 Bf(x,u) = \pm f'(x)\qquad\mbox{or}\qquad Bf(x,u) = f(x+u)-f(x).
\end{equation}
The first form of $B$ models a reflection process ($+$ forcing a reflection to the right and $-$ forcing a reflection to the left)
and the second form models a jump process jumping from $x$ to $x+u$. With these generators, a 
specification of the dynamics that requires
 \begin{equation}
 \label{introduction:mgp}
 f(X_t) - f(x_0)-\int_0^t \tilde{A}f(X_s,u_s)ds - \int_0^t Bf(X_s, u_s)d\xi_s
 \end{equation}
 to be a martingale for all $f\in C_c^2$ is equivalent to (\ref{introduction:sde}) in terms of weak solutions. Hence,
 the values of the cost criteria determined by (\ref{introduction:lta-criterion}) and (\ref{introduction:ihd-criterion}) remain identical.
 The following relaxed formulation of (\ref{introduction:mgp}) is better suited for the purpose of stochastic control.
 \begin{defn}
 \label{introduction:def-relaxed-mgp}
  Let $X$ be a stochastic process with state space $E$, let $\Lambda$ be a stochastic process taking values in $\mathcal{P}(U)$, and let 
$\Gamma$ be a random variable taking values in the space of measures on $([0,\infty)\times E \times U)$, with 
$\Gamma([0,t]\times E \times U)\in \mathcal{M}([0,t]\times E \times U)$ for all $t$. The triplet $(X, \Lambda, \Gamma)$ is a relaxed solution
to the singular, controlled martingale problem for $(\tilde{A},B)$ if there is a filtration $\{\mathscr{F}_t\}_{t\geq 0}$ such that 
$X$, $\Lambda$ and $\Gamma$ are $\mathscr{F}_t$-progressively measurable and
  \begin{equation*}
   f(X_t) - f(x_0) - \int_0^t \int_U \tilde{A}f(X_s,u)\,\Lambda_s(du)\,ds - \int_{[0,t]\times E \times U} Bf(x,u)\,\Gamma(ds\times dx\times du)
  \end{equation*}
  is an $\{\mathscr{F}_t\}_{t\geq 0}$-martingale for all $f\in C^2_c(E)$.
 \end{defn}
The relaxation is given by the fact that the control is no longer represented by a process $u$, but is encoded in the random measures $\Lambda$
and $\Gamma$. Assume that the cost functions $\tilde{c}_0$ and $\tilde{c}_1$ are continuous and non-negative. 
The cost criteria for a relaxed solution of the singular, controlled martingale problem are
 \begin{equation}
 \label{introduction:lta-criterion-relaxed}
 \limsup_{t\rightarrow \infty} \frac{1}{t}\mathbb{E}\left[\int_0^t\int_U \tilde{c}_0(X_s,u)\,\Lambda_s(du)\,ds
 + \int_{[0,t]\times E \times U} \tilde{c}_1(x,u)\,\Gamma(ds\times dx\times du) \right],
 \end{equation}
 for the long-term average cost criterion, and for $\alpha>0$,
 \begin{equation}
 \label{introduction:ihd-criterion-relaxed}
  \mathbb{E}\left[\int_0^\infty\int_U e^{-\alpha s} \tilde{c}_0(X_s,u)\,\Lambda_s(du)\,ds
 + \int_{[0,\infty)\times E \times U} e^{-\alpha s} \tilde{c}_1(x,u)\,\Gamma(ds\times dx\times du) \right],
 \end{equation}
 for the infinite horizon discounted cost criterion. A stochastic control problem given by 
 (\ref{introduction:mgp}) together with (\ref{introduction:lta-criterion-relaxed}) or (\ref{introduction:ihd-criterion-relaxed}) 
 can be reformulated as an infinite dimensional linear program. To this end, we set
 \begin{equation*}
  c_0(x,u) = \left\{\begin{array}{ll}\tilde{c}_0(x,u)&\mbox{ if }\alpha=0\\ \tilde{c}_0(x,u)/\alpha&\mbox{ if }\alpha>0\end{array}\right.
 \quad \mbox{and} \quad c_1(x,u) = \left\{\begin{array}{ll}\tilde{c}_1(x,u)&\mbox{ if }\alpha=0\\ \tilde{c}_1(x,u)/\alpha&\mbox{ if }\alpha>0.\end{array}\right.
 \end{equation*}
 Furthermore, for $\alpha\geq 0$ define the operator $A:C_c^2(E) \mapsto C(E\times U)$ by
 \begin{equation}
  \label{introduction:notation-and-formalism:form-of-a}
  Af(x,u) = \tilde{A}f(x,u)-\alpha f(x)
 \end{equation}
 and the functional $Rf = -\alpha f(x_0)$, $x_0$ being
the starting point of the diffusion.
 \begin{defn}
 \label{introduction:lp}
 The infinite-dimensional linear program for a singular stochastic control problem is given by
\begin{equation*} 
\def\arraystretch{2.5}
\begin{array}{r@{}l}
    \mbox{Minimize} \quad &{}\int_{E\times U} c_0 d\mu_0 + \int_{E\times U} c_1 d\mu_1\\[0.2cm]
\mbox{Subject to} \quad &{}\left\{\def\arraystretch{1.2} \begin{array}{l} 
                            \int_{E\times U} A f d\mu_0 + \int_{E\times U} Bf d\mu_1 = Rf\quad \forall f\in C_c^2(E)\\
                            \mu_0 \in \mathcal{P}(E \times U)\\
                            \mu_1 \in \mathcal{M}(E \times U).\\
                           \end{array}\right. 
\end{array}
\end{equation*}
\end{defn}
The measures $\mu_0$ and $\mu_1$ are the expected occupation measures of $X$ and $\Gamma$. We frequently consider the measures
on $\left(E,\mathscr{B}(E)\right)$ given by $\mu_{0,E}(\cdot) = \mu_0(\cdot\times U)$ and $\mu_{1,E}(\cdot) = \mu_1(\cdot\times U)$.
The refer to these measures as the state space marginals of $\mu_0$ and $\mu_1$, respectively.
The properties of such linear programs and their relation to stochastic control problems for singular, controlled martingale problems are stated
in Theorem \ref{introduction:equivalence-lp-lta-opt}. These results use the notion of
a regular conditional probability defined as follows.
\begin{defn}
 Let $(E\times U, \mathscr{B}(E\times U), \mu)$ be a measure space, and let $P: E\times U \ni (x,u) \mapsto x\in E$ be the projection map
 onto $E$. Let $\mu_E$ be the distribution of $P$, which is identical to the state space marginal of $\mu$. A map 
 $\eta: \mathscr{B}(U)\times E \mapsto [0,1]$ is called a regular conditional probability if
\begin{enumerate}[\hspace{0.6cm}i)]
 \item for each $x\in E$, $\eta(\cdot,x):\mathscr{B}(U) \mapsto [0,1]$ is a probability measure,
 \item for each $V \in \mathscr{B}(U)$, $\eta(V,\cdot): E\mapsto [0,1]$ is a measurable function, and
 \item for all $V\in\mathscr{B}(U)$ and all $F \in \mathscr{B}(E)$ we have
 \begin{equation*}
  \mu(F\times V) = \int_F \eta(V,x)\mu_E(dx).
 \end{equation*}
\end{enumerate}
\end{defn}
\begin{thm}
 \label{introduction:equivalence-lp-lta-opt}
   The problem of minimizing either the long-term average cost criterion of (\ref{introduction:lta-criterion-relaxed}) or the
   infinite horizon discounted cost criterion of (\ref{introduction:ihd-criterion-relaxed}) over the set of all
   relaxed solutions $(X,\Lambda,\Gamma)$ to the singular, controlled martingale problem for $(\tilde{A},B)$
   is equivalent to the linear program stated in (\ref{introduction:lp}). Moreover, there exists an optimal
   solution $(\mu_0^*,\mu_1^*)$. Let $\eta_0^*$ and $\eta_1^*$ be the regular conditional probabilities of $\mu_0^*$ and $\mu_1^*$ with
   respect to their state space marginals. Then an optimal relaxed control is given in feedback form by $\Lambda_t^* = \eta_0^*(\cdot, X^*_t)$ and 
   $\Gamma^*(ds\times dx \times du) = \eta_1^*(du,x)\tilde{\Gamma}^*(ds\times dx)$ for a random measure $\tilde{\Gamma}^*$ on $[0,\infty)\times E$,
   where $(X^*, \Lambda^*, \Gamma^*)$ is a relaxed solution to the singular, controlled martingale problem for $(\tilde{A},B)$ having occupation
   measures $(\mu_0^*, \mu_1^*)$.
\end{thm}
\begin{proof}
 See Kurtz and Stockbridge \cite{kurtz-stockbridge}, Theorem 2.1 and Theorem 3.3, respectively.
\end{proof}
By this result, it suffices to find optimal solutions to the infinite linear program when solving a singular stochastic control problem, and
approximate solutions to the linear program serve as approximate solutions to the control problem.
\Cref{discretization} presents how we discretize the infinite dimensional linear program to a computationally attainable formulation, 
which is the basis for
the numerical technique used in this paper. The analysis of this discretization scheme relies in part on the notion of weak convergence of
\textit{finite} measures which is defined next. Let $S$ be a measurable space in the following, equipped with a topology.
\begin{defn}
 Consider a sequence of finite measures $\{\mu_n\}_{n\in \mathbb{N}}$ and another finite measure $\mu$ on $S$. 
 We say that $\mu_n$ converges weakly to $\mu$, in symbol $\mu_n \Rightarrow \mu$, if for all $f\in C^u_b(S)$
 \begin{equation*}
 \int_S f(x) \mu_n(dx) \rightarrow \int_S f(x) \mu(dx)\quad\mbox{as }n\rightarrow \infty
 \end{equation*}
 holds.
\end{defn}
Note that we are considering finite measures, and not necessarily probability measures. In particular, we could encounter a situation
where the sequence of numbers $\{\mu_n(S)\}_{n\in\mathbb{N}}$ is unbounded. This differs from `classical' considerations of weak
convergence, which for example can be found in Billingsley \cite{billingsleyconvergence}. However, Bogachev \cite{bogachev-measure-theory}
(see Chapter 8 in Volume 2) offers a discussion of the concept of weak convergence in this more general case. Central to our purposes is
Theorem \ref{introduction:prokhorov}, which states sufficient conditions for the existence of weakly converging subsequences when
considering sequences of finite measures, based on the following two
concepts.
\begin{defn}
\label{introduction:defn-tight}
 A sequence of finite measures $\{\mu_n\}_{n\in \mathbb{N}}$ on $S$ is called tight if for each $\epsilon>0$, there is a compact set 
 $K_\epsilon$ in $S$ such that 
 \begin{equation*}
 \mu_n(K_\epsilon^C) < \epsilon
 \end{equation*}
 holds for all $n\in \mathbb{N}$.
\end{defn}
\begin{rmk}
\label{introduction:compact-remark}
 If $S$ is compact, any sequence of finite measures on $S$ is tight.
\end{rmk}
\begin{defn}
\label{introduction:defn-unifbound}
 A sequence of finite measures $\{\mu_n\}_{n\in \mathbb{N}}$ on $S$ is called uniformly bounded if for some $l\geq 0$, $\mu_n(S)\leq l$
 holds for all $n\in \mathbb{N}$. 
\end{defn}
If a sequence of finite measures on $S$ is tight and uniformly bounded, the existence of convergent subsequences is guaranteed by the 
following result.
\begin{thm}
\label{introduction:prokhorov}
 Let $\{\mu_n\}_{n\in \mathbb{N}}$ be a sequence of finite measures on $S$. Then, the following are equivalent.
 \begin{enumerate}[i)]
  \item $\{\mu_n\}_{n\in \mathbb{N}}$ contains a weakly convergent subsequence,
  \item $\{\mu_n\}_{n\in \mathbb{N}}$ is tight and uniformly bounded.
 \end{enumerate}
\end{thm}
\begin{proof}
 See Bogachev \cite{bogachev-measure-theory}, Theorem 8.6.2.
\end{proof}

\section{Approximation}
We begin the presentation of the proposed method by describing the discretization scheme. Then, we discuss the assumption being necessary
for the convergence of the method.
\label{approximation-section}
\subsection{Discretization}
\label{discretization}
The proposed numerical technique is based on a discretization of the infinite-dimensional linear program in three steps.
First, we introduce a limit on the full mass of the measure $\mu_1$. Then, we restrict the number of constraint functions. Thirdly, we introduce 
discrete versions of the measures. In the process, several assumptions on the measure $\mu_0$ are made. For the sake of exposition,
we elaborate on these assumptions separately in \Cref{assumptions}.\\
Since the discretization brings forth several distinct sets of measures, we define the cost criterion 
using the following, general formulation.
\begin{equation*}
J: \mathcal{P}(E\times U) \times \mathcal{M}(E\times U) \ni(\mu_0,\mu_1) \mapsto J(\mu_0,\mu_1) = \int_{E\times U} c_0 d\mu_0 + \int_{E\times U} c_1 d\mu_1 \in \mathbb{R}_{+}
\end{equation*}
We choose to consider $C_c^2(E)$ as a normed space in the right sense. Set $\|f\|_{\mathscr{D}} = \|f\|_\infty +\|f'\|_\infty + \|f''\|_\infty$
and define $\mathscr{D}_\infty = (C_c^2(E), \|\cdot \|_{\mathscr{D}})$ to designate that we consider $C_c^2(E)$ to be a specific normed space.
Set
\begin{equation*}
 \mathscr{M}_\infty = \{ (\mu_0, \mu_1) \in \mathcal{P}(E\times U) \times \mathcal{M}(E\times U) : 
 \int A f d\mu_0 + \int Bf d\mu_1 = Rf \quad\forall f \in \mathscr{D}_\infty\}.
\end{equation*}
For analytical purposes, we introduce an upper bound on $\mu_1(E\times U)$. For $l>0$ define
\begin{equation}
\label{approximation:ml-mass-restriction}
 \mathscr{M}^l_\infty = \{(\mu_0, \mu_1) \in \mathscr{M}_\infty : \mu_1(E\times U)\leq l \}.
\end{equation}
\begin{rmk}
\label{approximation:rmk-on-lbound}
As $l$ increases, more measures of $\mathscr{M}_\infty$ will lie in $\mathscr{M}^l_\infty$. For $l$ large enough, the optimal
solution will lie in $\mathscr{M}^l_\infty$, as we have that $\mu^*_1\in \mathcal{M}(E\times U)$ and hence $\mu^*_1(E)<\infty$.
\end{rmk}
\begin{defn}
The $l$-bounded infinite-dimensional linear program is given by\\
\begin{equation*}
\min\left\{J(\mu_0,\mu_1) | (\mu_0,\mu_1) \in \mathscr{M}^l_\infty \right\}.
\end{equation*}
\end{defn}
\noindent The set $\mathscr{M}^l_\infty$ features an infinite set of constraints given by all $f\in \mathscr{D}_\infty$ and 
measures $\mu_0$ and $\mu_1$ having an infinite number of degrees of freedom.
First, we discretize the set of constraints using B-spline basis functions. To construct these basis
functions, fix $q\in \mathbb{N}$ and consider a finite set of pointwise distinct grid points $\{e_k\}_{k=-3}^{q+3}$ in $E$, with $e_0=e_l$,
$e_r=e_q$ and $e_{k}<e_{k+1}$ for $k=-3,-2,\ldots,q+2$.
\begin{defn}
 The set of cubic B-spline basis functions for a grid $\{e_k\}_{k=-3}^{q+3}$ is defined on $\mathbb{R}$ by
 \begin{equation*}
 f_k(x) = (e_{k+4}-e_k) \sum_{i=k}^{k+4}\frac{\left[(e_i-x)^3\right]^+}{\Psi'_k(e_i)},\quad k=-3,-2,\ldots, q-1,
 \end{equation*}
 where
 \begin{equation*}
 \Psi_k(x) = \prod_{i=k}^{k+4} (x-e_i),\quad k=-3,-2,\ldots,q-1.
 \end{equation*}
\end{defn}
An analysis of these basis function is given in de Boor \cite{deboorsplines}. Provided that 
\begin{equation*}
\displaystyle \max_{k=-3,\ldots,q+2}\;\left(e_{k+1}-e_k\right)\rightarrow 0\quad\mbox{and}\quad
\frac{\displaystyle \max_{k=-3,\ldots,q+2}\;\left(e_{k+1}-e_j\right)}{\displaystyle \min_{k=-3,\ldots,q+2}\;\left(e_{k+1}-e_j\right)}\rightarrow 1
\end{equation*}
as $n\rightarrow \infty$, Theorem 1 of Hall and Meyer \cite{cubic-spline-approx} holds and the following statement can be shown.
\begin{prop}
\label{approximationm:dinf-separable}
 The normed space $\mathscr{D}_\infty$ is separable and a
countable basis\\ $\{f_k\}_{k\in\mathbb{N}}$ is given by the cubic B-splines basis functions.
\end{prop}
 For fixed $\tilde{q}\in\mathbb{N}$, define a grid using the dyadic partition of $E$ given by
 \begin{equation*} 
  e_k = e_l + \frac{e_r-e_l}{2^{\tilde{q}}}\cdot k,\qquad k=-3\ldots, 2^{\tilde{q}}+3.
 \end{equation*}
and consider the $n:=2^{\tilde{q}}+2$ B-spline basis functions $\{f_k\}_{k=1}^n$ on this grid. This allows us to define
\begin{align*}
 \mathscr{M}_n &= \left\{ (\mu_0, \mu_1) \in \mathcal{P}(E\times U) \times \mathcal{M}(E\times U):\right.\\
	  &\left.\qquad\qquad\qquad \int A f_k d\mu_0 + \int Bf_k d\mu_1 = Rf_k,\quad k=1,\ldots,n\right\}
\end{align*}
and we can define $\mathscr{M}_n^l$ in a similar manner using the mass restriction on $\mu_1$ as seen in
(\ref{approximation:ml-mass-restriction}).
\begin{defn}
The $l$-bounded $(n,\infty)$-dimensional linear program is given by
\begin{equation*}
\min\left\{J(\mu_0, \mu_1) | (\mu_0, \mu_1) \in \mathscr{M}^l_n \right\}.
\end{equation*}
\end{defn}
\noindent Next, we discretize the measures. Theorem \ref{introduction:equivalence-lp-lta-opt} reveals that it is
sufficient to regard feedback controls which can be represented by regular conditional probabilities.
In particular, this result states that we can consider measures $(\mu_0,\mu_1)$ which can be decomposed according to
$\mu_0(dx\times du) = \eta_0(dx,u)\mu_{0,E}(dx)$ and $\mu_1(dx\times du) = \eta_1(dx,u)\mu_{1,E}(dx)$ 
for two regular conditional probabilities $\eta_0$ and $\eta_1$.
We furthermore assume that, first, for any interval or singleton $V \subset U$, $x\mapsto \eta_0(V,x)$ is continuous almost everywhere with respect to Lebesgue
measure, second, that $\mu_{0,E}$ has a density $p$ with respect to Lebesgue measure and third, that $p$ satisfies the constraint that 
$\lambda\left(\left\{x : p(x) = 0 \right\} \right)=0$. In other words, $p$ must only be equal to zero on a set of Lebesgue measure $0$.
The particulars of these assumptions are discussed in \Cref{assumptions}, and we continue here with the description of the approximation scheme.\\
Define a sequence $k_m$ as follows. As $c_0$, $b$ and $\sigma$ are continuous over a
compact set, for all $m\in\mathbb{N}$, there is a $\delta_m>0$ such that for all $u,v\in U$  with $\vert u-v \vert\leq \delta_m$, it is true that
\begin{equation}
\label{convergence:discrete-U-bounds}
\max\left\{\vert c_0(x,u) - c_0(x,v)\vert,\vert b(x,u)-b(x,v)\vert,\left\vert \frac{1}{2}\sigma^2(x,u)-\frac{1}{2}\sigma^2(x,v)\right\vert\right\} \leq\frac{1}{2^{m+1}},
\end{equation}
uniformly in $x$.
Set $k_m$ to be the smallest integer such that $\frac{u_r-u_l}{2^{k_m}}\leq\delta_m$.
The parameter $k_m$ controls the discretization of the control space $U$, and the specific choice enables an accurate approximate integration of
the cost function $c_0$ and the functions $Af_k$ against the relaxed control $\eta_0$ in the convergence proof of \Cref{convergence}.
So, define
\begin{equation}
\label{convergence:discrete-U-set-points}
U^{(m)} = \{ u_j = u_l + \frac{u_r-u_l}{2^{k_m}}\cdot j, j=0,\ldots, 2^{k_m}\}.
\end{equation}
Similarly, we set
\begin{equation}
\label{convergence:discrete-E-set-points}
E^{(m)} = \{ e_j = e_l + \frac{e_r-e_l}{2^{m}}\cdot j, j=0,\ldots, 2^{m}\}.
\end{equation}
The union of these sets over all $m\in\mathbb{N}$ is dense in the control space and state space, respectively.
The number $m$ is called the discretization level. It determines the degrees of freedom of the discrete measures 
$\hat{\mu}_0$ (approximating $\mu_0$) and $\hat{\mu}_1$ (approximating $\mu_1$), which are defined as follows.\\
First, we approximate the density $p$ of $\mu_{0,E}$. Choose a countable basis of $L^1(E)$, say $\{p_n\}_{n\in \mathbb{N}_0}$, given by indicator functions over subintervals of $E$.
We truncate this basis to $p_0,\ldots,p_{2^{m}-1}$ (given by the indicator functions of the intervals of length $1/2^{m}$, compare
(\ref{convergence:discrete-E-set-points})) to approximate the density $p$ by
\begin{equation}
\hat{p}_{m}(x) = \sum_{j=0}^{2^{m}-1} \gamma_j p_j(x) \label{approximation:mnk_density}
\end{equation}
where $\gamma_j\in \mathbb{R}_+$, $j=0,\ldots, 2^{m}-1$ are weights to be chosen under the constraint that\\
$\int_E \hat{p}_{m}(x) dx=1$. 
Set $E_j=[x_j,x_{j+1})$ for $j=0,1,\ldots 2^{m}-2$ and $E_{2^{m}-1} = [x_{2^{m}-1}, x_{2^{m}}]$ to define 
\begin{equation}
\label{approximation:mnk_kernel_c}
\hat{\eta}_{0,m}\left(V, x\right) = \sum_{j=0}^{2^{m}-1}\sum_{i=0}^{2^{k_m}} \beta_{j,i}I_{E_j}(x) \delta_{\{u_i\}}(V),
\end{equation}
where $\beta_{j,i} \in \mathbb{R}_{+}$, $j=0,\ldots, m-1, i=0, \ldots, k_m$ are 
weights to be chosen under the constraint that $\sum_{i=0}^{2^{m}-1} \beta_{j,i}=1$ for $j=0,\ldots,m-1$. We approximate
$\eta_0$ using (\ref{approximation:mnk_kernel_c}), which means that this
relaxed control is approximated by point masses in $U$-`direction' and piecewise constant in 
$E$-`direction'. Then, we set
$\hat{\mu}_{0,m}(du\times dx) = \hat{\eta}_{0,m}(du,x)\hat{p}_{m}(x)dx$.\\
To approximate the singular occupation measure $\mu_1$, we use that the process is only showing singular behavior at $e_l$ and $e_r$.
Thus, if we introduce the regular conditional probability $\eta_1$ and write $\mu_1(dx \times du) = \eta_1(du,x) \mu_{1,E}(dx)$, and for 
$F\in \mathscr{B}(E)$, we have for $F\in\mathscr{B}(E)$
\begin{equation}
\mu_{1,E}(F) = w_1 \delta_{\{e_l\}}(F) + w_2 \delta_{\{e_r\}}(F)\label{approximation:inft_bss_bcs:mnk_singular}
\end{equation}
with $w_1,w_2\in \mathbb{R}_+$. We approximate the relaxed control $\eta_1$ by 
\begin{equation}
\label{approximation:mnk_kernel_s}
\hat{\eta}_{1,m}(V,e_l) = \sum_{i=0}^{2^{k_m}}\zeta_{1,i} \delta_{\{u_i\}}(V),\quad
\hat{\eta}_{1,m}(V,e_r) = \sum_{i=0}^{2^{k_m}}\zeta_{2,i} \delta_{\{u_i\}}(V)
\end{equation}
with $\sum_{i=0}^{2^{k_m}}\zeta_{j,i}=1$ for $j=1,2$.
So, we have $\hat{\mu}_{1,m}(dx\times du) = \hat{\eta}_{1,m}(du,x)\mu_{1,E}(dx)$. In summary, we consider measures of the form
\begin{equation*}
\left(\hat{\mu}_{0,m}, \hat{\mu}_{1,m}\right)(dx\times du) = \left(\hat{\eta}_{0,m}(du,x)\hat{p}_{m}(x)dx ,\hat{\eta}_{1,m}(du,x)\mu_{1,E}(dx)\right)
\end{equation*}
and we introduce the notation
\begin{multline*}
\mathscr{M}^l_{n,m} = \big\{(\mu_{0,m}, \mu_{1,m}) \in \mathscr{M}^l_n : \left(\mu_{0,m}, \mu_{1,m}\right)(du,dx)\\
= \left(\hat{\eta}_{m}(du,x)\hat{p}_{m}(x)dx ,\hat{\eta}_{m}(du,x)\mu_{1,E}(dx)\right) \big\}.
\end{multline*}
This finalizes the discretization of the measures and leaves us with the following linear program.
\begin{defn}
The $l$-bounded $(n,m)$-dimensional linear program is given by
\begin{equation*}
\inf\left\{J(\mu_0,\mu_1) | (\mu_0,\mu_1) \in \mathscr{M}^l_{n,m} \right\}.
\end{equation*}
\end{defn}
This linear program is linear in the coefficients given by the products $\beta_{j,i}\cdot \gamma_j$ and $\zeta_{j,i}\cdot \alpha_j$, and the cost functional can as well
be expressed as a linear combination of these coefficients.\\
Up to this point, we introduced four sets of measures, $\mathscr{M}_\infty$, $\mathscr{M}^l_\infty$, $\mathscr{M}^l_n$ and $\mathscr{M}^l_{n,m}$,
and we later on will use $\epsilon$-optimal solutions in $\mathscr{M}^l_{n,m}$ to approximate the optimal solution in $\mathscr{M}_\infty$.
However the relations between those sets are $\mathscr{M}^l_\infty\subset \mathscr{M}_\infty$,
$\mathscr{M}^l_n \supset \mathscr{M}^l_\infty$ and $\mathscr{M}^l_{n,m} \subset \mathscr{M}^l_n$. As this does not provide a clear nested
structure, it has to be carefully analyzed how optimal solutions in these sets relate to each other. This is presented in
\Cref{convergence}.

\subsection{Assumptions}
\label{assumptions}
Before we move to the presentation of the convergence argument, we elaborate on the assumptions on $\mu_0$ which were made 
in \Cref{discretization}.
These assumptions restrict the set of feasible measures considered in the linear program given by (\ref{introduction:lp}) 
to measures which allow the approximation to converge. 
Albeit technical, the imposed restrictions do not curtail the set of feasible measures beyond what can be considered to be 
`implementable' solutions, in other words, the set of measures will still be large enough to include any type of control that could be used in a
real-world application.\\
First, we assume that the state space marginal $\mu_{0,E}$ of the expected occupation measures $\mu_0$ has a density $p$ with respect to the 
Lebesgue measure. As shown in \cite{mgvieten-thesis}, Section II.2, this is guaranteed when certain assumptions on the regular conditional probability 
$\eta_0$ of the continuous occupation measure $\mu_0$ are fulfilled. To be precise, we have to assert that the functions
\begin{equation}
x \mapsto \int_U b(x,u)\,\eta_0(du,x),\qquad
x \mapsto \int_U \sigma^2(x,u)\,\eta_0(du,x)\label{assumptions-when-density}\\
\end{equation}
are continuous everywhere except for finitely many points in $E$. On the one hand, this is satisfied for controls of the form given by
(\ref{approximation:mnk_kernel_c}), which includes the important class of so-called bang-bang controls. Bang-bang controls put full mass 
on either of the end points $u_l$ and $u_r$ of the control space $U$. Usually, when the cost function $c_0$ does not depend on the control
value $u$, the optimal solution is given by a bang-bang control. If this is not the case, optimal controls are frequently
given in the form of a continuous function 
$v:E\mapsto U$ and a control satisfying $\eta_0(\left\{v(x)\right\},x)=1$. It is easy to see that in both cases the two functions defined in
(\ref{assumptions-when-density}) are continuous except for finitely many points.\\
Secondly, we assume that $p$ must be equal to zero only on a set of Lebesgue measure $0$. The analysis in \cite{mgvieten-thesis}, Section II.2
shows that the densities encountered when using both the long-term average cost criterion and the discounted infinite horizon criterion satisfy
this assumption.
Thirdly, we assume
that for any set $V\subset U$ which is either an interval or a singleton, the function $x\mapsto \eta_0(V,x)$ is continuous almost everywhere with
respect to Lebesgue measure. This allows us to approximate the function 
$x\mapsto \eta_0(V,x)$ uniformly by a function which is piecewise constant over intervals, and the approximate function values on these intervals
are given by the values of $\eta_0(V,x)$ at the left endpoints of the intervals. This makes the statement of 
(\ref{convergence:approximation_of_control}) of the convergence argument true.
Controls of the form given in (\ref{approximation:mnk_kernel_c}) satisfy this requirement.
However, if $\eta_0$ fulfills $\eta_0(\left\{v(x)\right\},x)=1$ for some continuous function $v$ on $E$, we have to assert
more regularity on $v$, according to the following definition.
\begin{defn}
 A continuous function $v:E\mapsto U$ is said to have finitely many modes if there are finitely many points
 $e_l=\hat{y}_1< \hat{y}_{2} < \ldots <\hat{y}_k=e_r$ such that for all $2\leq i \leq k-1$, there are points $a_i$ and $b_i$ with
 $\hat{y}_{i-1}<a_i<\hat{y}_i<b_i<\hat{y}_{i+1}$ and either of the following statements hold:
 \begin{enumerate}[i)]
  \item $v$ is strictly increasing on $(a_i,\hat{y}_i)$ and strictly decreasing on $(\hat{y}_i, b_i)$ as well as
  increasing on $(\hat{y}_{i-1},\hat{y}_i)$ and decreasing on $(\hat{y}_i, \hat{y}_{i+1})$,
  \item $v$ is strictly decreasing on $(a_i,\hat{y}_i)$ and strictly increasing on $(\hat{y}_i, b_i)$ as well as
  decreasing on $(\hat{y}_{i-1},\hat{y}_i)$ and increasing on $(\hat{y}_i, \hat{y}_{i+1})$.
 \end{enumerate}
\end{defn}
We assume that $v$ only has finitely many modes in the following. The rationale behind this assumption is as follows. Obviously, 
$x\mapsto \eta_0(V,x)$ is piecewise constant, either $0$ or $1$. The fact that $v$ `oscillates' only finitely many times between its 
modes ensures that $x\mapsto \eta_0(V,x)$ does not switch from $0$ to $1$ or from $1$ to $0$ more than finitely many times, and hence it
is discontinuous on a set that has measure $0$ with respect to Lebesgue measure. \\

\section{Convergence}
The first part of this section gives an overview of the convergence argument, illustrating the main ideas of the analysis.
The proofs of the propositions and corollary are given in the second part.
\label{convergence}
\subsection{Statement of the main results}
\label{convergence-main-results}
The $l$-bounded $(n,m)$-dimensional linear program introduced in \Cref{discretization} is a finite dimensional linear program that can be solved with 
standard solvers that are available in numerical libraries, and hence optimal solutions are attainable. We proceed
to show that the optimal solution to the $l$-bounded $(n,m)$-dimensional linear program is an $\epsilon$-optimal solution to the infinite dimensional program for $l$, $n$ and $m$
large enough. We use the notations 
\begin{align*}
&{ }J^{*}\quad=\inf\left\{J(\mu_0,\mu_1) | (\mu_0,\mu_1) \in \mathscr{M}^l_\infty \right\}\\
&{ }J^{*}_n\quad=\inf\left\{J(\mu_0,\mu_1) | (\mu_0,\mu_1) \in \mathscr{M}^l_{n} \right\}\\
&{ }J^{*}_{n,m}\,= \inf\left\{J(\mu_0,\mu_1) | (\mu_0,\mu_1) \in \mathscr{M}^l_{n,m} \right\}.
\end{align*}
For $l$ large enough, $J^*$ is indeed the optimal solution
to the unbounded problem, as stated in Remark \ref{approximation:rmk-on-lbound}, in other words,
$J^*=\min\left\{J(\mu_0,\mu_1) | (\mu_0,\mu_1) \in \mathscr{M}_\infty \right\}$\\
Since an infimum might not be computationally
attainable in $\mathscr{M}^l_\infty$ and $\mathscr{M}^l_{n}$, we withdraw to the slightly relaxed optimization problem of finding an 
$\epsilon$-optimal solution, in other
words, we try to find a pair of measures $(\mu_0^\epsilon, \mu_1^\epsilon)\in \mathscr{M}^l_\infty$ such that
\begin{equation*} 
J(\mu_0^\epsilon, \mu_1^\epsilon)-J(\mu_0,\mu_1) \leq \epsilon \quad \forall (\mu_0,\mu_1)\in \mathscr{M}^l_\infty.
\end{equation*}
Note that trivially, $J(\mu_0^\epsilon, \mu_1^\epsilon)-J^*\geq 0$. The $\epsilon$-optimality for $\mathscr{M}^l_{n}$ is defined analogously.
The following convergence analysis proves that we can find $\epsilon$-optimal measures in $\mathscr{M}^l_\infty$ using the approximation
proposed in \Cref{approximation-section}.
The outline of the proof is as follows. First, it is shown that it suffices to find an $\epsilon$-optimal solution in 
$\mathscr{M}^l_{n}$.
\begin{prop}
\label{convergence:proposition1}
 For each $n\in\mathbb{N}$, assume that $(\mu_{0,n}^\epsilon,\mu_{1,n}^\epsilon)\in \mathscr{M}^l_n$ and that for $n\in\mathbb{N}$, 
 $(\mu_{0,n}^\epsilon,\mu_{1,n}^\epsilon)$ is an
 $\epsilon$-optimal solution for the $l$-bounded, $(n,\infty)$-dimensional linear program. Then, for $\delta>0$, there exists an $N(\delta)$ such that
 \begin{equation*}
 J(\mu_{0,n}^\epsilon,\mu_{1,n}^\epsilon) - J^*  \leq 2 \epsilon + \delta.
 \end{equation*}
 for all $n\geq N(\delta)$.
\end{prop}
Next, we establish that $\epsilon$-optimal solutions in $\mathscr{M}^l_{n}$ can be obtained using the discretization
introduced in \Cref{discretization}. The central result reads as follows.
\begin{prop} 
For $(\mu_0,\mu_1) \in \mathscr{M}^l_n$ and each $\epsilon>0$, there is an $m_0$ such that for all $m\geq m_0$ there exists a 
$(\hat{\mu}_{0,m}, \hat{\mu}_{1,m}) \in \mathscr{M}^l_{n,m}$,
with
\begin{equation*}
\left\vert J(\mu_0,\mu_1) - J(\hat{\mu}_{0,m}, \hat{\mu}_{1,m})\right \vert <\epsilon.
\end{equation*}
\label{convergence:second_part_big_proposition}
\end{prop}
This result shows that arbitrary (not necessarily optimal) measures in $\mathscr{M}^l_n$ can be approximated, in terms of their
cost criterion, by measures in $\mathscr{M}^l_{n,m}$. Regarding optimal measures, the following statement is an easy consequence
from Proposition \ref{convergence:second_part_big_proposition}.
\begin{crl}
\label{convergence:second_part_corollary}
For each $m\in \mathbb{N}$, assume that $(\mu^*_{0,n,m}, \mu^*_{1,n,m})\in \mathscr{M}^l_{n,m}$ and that for $m\in \mathbb{M}$, 
$(\mu^*_{0,n,m}, \mu^*_{1,n,m})$ is an optimal solution to the $l$-bounded, $(n,m)$-dimensional linear program.
 Then, the sequence of numbers $\{J(\mu^*_{0,n,m}, \mu^*_{1,n,m})\}_{m\in\mathbb{N}}$\\
 converges to $J^*_n$ as $m\rightarrow \infty$.
\end{crl}
In conjunction, the preceding results allow us to prove the following theorem.
\begin{thm}
\label{main-theorem}
 For any $\epsilon>0$, there is an $l>0$, an $N\in\mathbb{N}$ and an $M\in\mathbb{N}$ such that 
  \begin{equation*}
 \vert J^* - J^*_{n,m} \vert < \epsilon
 \end{equation*}
 holds for all $n\geq N$ and $m\geq M$.
\end{thm}
\begin{proof}
 Pick any $\epsilon>0$. Choose $l$ large enough such that\\
$J^*=\min\left\{J(\mu_0,\mu_1) | (\mu_0,\mu_1) \in \mathscr{M}_\infty \right\}$. Pick $\hat{\epsilon}$ and $\delta>0$ 
in such a way that $2\hat{\epsilon}+\delta<\epsilon$. By 
 Proposition \ref{convergence:proposition1}, choose $N\in\mathbb{N}$ large enough such that for all $n\geq N$, an 
 $\hat{\epsilon}$-optimal solution $(\mu^{\hat{\epsilon}}_{0,n}, \mu^{\hat{\epsilon}}_{1,n})$ to the $l$-bounded, 
 $(n,\infty)$-dimensional program is an $2\hat{\epsilon}+\delta$-optimal solution to the $l$-bounded infinite-dimensional linear program.
 Now, using Corollary \ref{convergence:second_part_corollary}, choose $M\in\mathbb{N}$ large enough such that for all $m\geq M$, the optimal solution $(\mu^*_{0,n,m}, \mu^*_{1,n,m})$
 to the $l$-bounded $(n,m)$-dimensional linear program is an $\hat{\epsilon}$-optimal solution to the $l$-bounded $(n,\infty)$-dimensional
 program. But then,
 \begin{equation*}
  \vert J^* - J^*_{n,m} \vert \equiv \vert J^* - J(\mu^*_{0,n,m}, \mu^*_{1,n,m}) \vert < 2\hat{\epsilon}+\delta <\epsilon
 \end{equation*}
\end{proof}
\subsection{Proofs of the main results}
The proof of Proposition \ref{convergence:proposition1} is a rather straight forward application of the theory of weak convergence as introduced in 
\Cref{introduction:notation-and-formalism}. The proofs of Proposition \ref{convergence:second_part_big_proposition} and 
Corollary \ref{convergence:second_part_corollary} on the other hand require an in-depth analysis of the approximation properties of the proposed 
discretization scheme. We begin with the proof of Proposition \ref{convergence:proposition1}, stated again for the sake of exposition. \\
\setcounter{thm}{0}
\begin{prop}
 For each $n\in\mathbb{N}$, assume that $(\mu_{0,n}^\epsilon,\mu_{1,n}^\epsilon)\in \mathscr{M}^l_n$ and that for $n\in\mathbb{N}$, 
 $(\mu_{0,n}^\epsilon,\mu_{1,n}^\epsilon)$ is an
 $\epsilon$-optimal solution for the $l$-bounded, $(n,\infty)$-dimensional linear program. Then, for $\delta>0$, there exists an $N(\delta)$ such that
 \begin{equation*}
 J(\mu_{0,n}^\epsilon,\mu_{1,n}^\epsilon) - J^* \leq 2 \epsilon + \delta.
 \end{equation*}
 for all $n\geq N(\delta)$.
\end{prop}
\begin{proof} Assume first that $(\mu^\epsilon_{0,n},\mu^\epsilon_{1,n})$ converges weakly to some $(\mu_0^\epsilon,\mu_1^\epsilon)$. Then, 
for $g\in\mathscr{D}_\infty$, and $g_k\rightarrow g$ in $\mathscr{D}_\infty$ we have
\begin{align*}
\int Ag d\mu_0 + \int Bg d\mu_1 &= \lim_{k\rightarrow \infty} \lim_{n\rightarrow \infty} \int Ag_k d\mu^\epsilon_{0,n} + \int Bg_k d\mu^\epsilon_{1,n}\\
&= \lim_{k\rightarrow \infty} \lim_{n\rightarrow \infty} Rg_k = \lim_{k\rightarrow \infty} Rg_k = Rg.
\end{align*}
Thus, $(\mu_0^\epsilon,\mu_1^\epsilon)\in \mathscr{M}_\infty$. Now observe that since $c_1$ and $c_2$ are continuous,\\
$J(\mu^\epsilon_{0,n},\mu^\epsilon_{1,n})\rightarrow J(\mu_0^\epsilon,\mu_1^\epsilon)$. Assume there would be an
$(\hat{\mu}_0,\hat{\mu}_1) \in \mathscr{M}^l_\infty$ such that $J(\mu_0^\epsilon,\mu_1^\epsilon) > J(\hat{\mu}_0,\hat{\mu}_1) + \epsilon$,
that is, assume $(\mu_0^\epsilon,\mu_1^\epsilon)$ would not be $\epsilon$-optimal. For $n$ large enough, we would have 
$J(\mu_{0,n}^\epsilon,\mu_{1,n}^\epsilon) > J(\hat{\mu}_0,\hat{\mu}_1) + \epsilon$. But as $\mathscr{M}_n \supset \mathscr{M}_\infty$, 
this would imply that $(\mu_{0,n}^\epsilon,\mu_{1,n}^\epsilon)$ is not $\epsilon$-optimal. So, $(\mu_0^\epsilon,\mu_1^\epsilon)$ is
$\epsilon$-optimal.\\
In general, we cannot guarantee weak convergence of $(\mu^\epsilon_{0,n},\mu^\epsilon_{1,n})$, but since $E\times U$ is compact and the
full mass of $\mu^\epsilon_{0,n}$ and $\mu^\epsilon_{1,n}$ is uniformly bounded by $1$ and $l$, respectively, the existence
of a convergent subsequence is given by Theorem \ref{introduction:prokhorov}. Consider two different subsequences of $\epsilon$-optimal measures 
$(\mu^\epsilon_{0,n_1},\mu^\epsilon_{1,n_1})$ and
$(\mu^\epsilon_{0,n_2},\mu^\epsilon_{1,n_2})$. Set $z_1 = \lim_{n_1\rightarrow \infty}J(\mu^\epsilon_{0,n_1},\mu^\epsilon_{1,n_1})$ and 
$z_2 = \lim_{n_2\rightarrow \infty}J(\mu^\epsilon_{0,n_2},\mu^\epsilon_{1,n_2})$. Both $z_1$ and $z_2$ are $\epsilon$-optimal cost criterion 
values in $\mathscr{M}^l_\infty$.  Hence we can conclude that $\vert z_1-z_2\vert<\epsilon$. In particular, for $z\in\mathbb{R}$ such that 
$J(\mu^\epsilon_0,\mu^\epsilon_1)\in [z-\frac{\epsilon}{2}, z+\frac{\epsilon}{2}]$ for any weak limit $(\mu^\epsilon_0,\mu^\epsilon_1)$ of a
sequence of $\epsilon$-optimal measures.
This means that for $\delta>0$, there is an $N\equiv N(\delta)$ large enough such that for all 
$n\geq N$, $J(\mu^\epsilon_{0,n},\mu^\epsilon_{1,n})\in\left(z-\frac{\epsilon}{2}-\delta, z+\frac{\epsilon}{2}+\delta \right)$.
Now assume that $(\mu^\epsilon_{0,n},\mu^\epsilon_{1,n})$ does not converge weakly, and that 
$J(\mu^\epsilon_{0,n},\mu^\epsilon_{1,n}) \notin \left(z-\frac{\epsilon}{2}-\delta, z+\frac{\epsilon}{2}+\delta \right)$ infinitely often. 
This would allow for the construction of a subsequence $(\mu^\epsilon_{0,n_3},\mu^\epsilon_{1,n_3})$ with
$J(\mu_{n_3},\mu_{1,n_3}) \notin \left(z-\frac{\epsilon}{2}-\delta, z+\frac{\epsilon}{2}+\delta \right)\:\forall\, n_3\in\mathbb{N}$. However, by
the tightness argument of Theorem \ref{introduction:prokhorov}, that sequence contains a converging sub-subsequence $(\mu_{0,n_4},\mu_{1,n_4})$. 
This means that
$J(\mu_{0,n_4},\mu_{1,n_4})\in\left(z-\frac{\epsilon}{2}-\delta, z+\frac{\epsilon}{2}+\delta \right)$ eventually, contradicting the preceding
assumption. So, there exists an $N\in \mathbb{N}$ such that
$J(\mu^\epsilon_{0,n},\mu^\epsilon_{1,n})\in\left(z-\frac{\epsilon}{2}-\delta, z+\frac{\epsilon}{2}+\delta \right)\:\forall\,n\geq N$.
Consider a limit $(\mu^\epsilon_0, \mu^\epsilon_1)$ of a convergent subsequence of $(\mu^\epsilon_{0,n},\mu^\epsilon_{1,n})$.
By its $\epsilon$-optimality and the properties of the infimum,
\begin{align*}
J^* + \epsilon &\geq J(\mu_0^\epsilon,\mu_1^\epsilon) \geq z-\frac{\epsilon}{2}\\
\Leftrightarrow J^* &\geq z-\frac{3\epsilon}{2}.
\end{align*}
Thereby, for $n\geq N$,
\begin{align*}
J(\mu^\epsilon_{0,n},\mu^\epsilon_{1,n}) - J^* \leq z + \frac{\epsilon}{2} + \delta - \left(z-\frac{3\epsilon}{2}\right) &= 2\epsilon +\delta\\
\end{align*}
from which the claim follows.
\end{proof}
\setcounter{thm}{5}

We proceed to analyze the discretization scheme of \Cref{discretization} and its approximation properties in order to prove
Proposition \ref{convergence:second_part_big_proposition}. From here on, we consider a fixed $n$, assuming that it is large enough
to guarantee that the statement of Proposition \ref{convergence:proposition1} holds. In particular, we consider the space $\mathscr{D}_n$, 
which is spanned by the basis functions $f_1, f_2,\ldots, f_n$.\\
Proposition \ref{convergence:second_part_big_proposition} states that we can approximate the cost
criterion of a pair of measures $(\mu_{0}, \mu_{1})\in\mathscr{M}^l_n$ arbitrarily closely by a pair of measures 
$(\mu_{m,0}, \mu_{m,1})\in\mathscr{M}^l_{n,m}$. In the following, we consider a fixed pair of measures
$(\mu_0,\mu_1)$. If $\eta_0$ and $\eta_1$ are the regular conditional probabilities of $\mu_0$ and
$\mu_1$, respectively, we use the following choice of the coefficients $\beta_{j,i}$ and $\zeta_{j,i}$ in
(\ref{approximation:mnk_kernel_c}) and (\ref{approximation:mnk_kernel_s}), respectively. We use the mesh points in $E$ and $U$ as
defined in (\ref{convergence:discrete-E-set-points}) and (\ref{convergence:discrete-U-set-points}). Set
\begin{equation}
\label{convergence:discrete-U-set}
U_i = [u_{i}, u_{i+1})\mbox{ for } 0\leq i \leq 2^{k_m}-1, \quad U_{2^{k_m}} = \{u_{2^{k_m}}\}
\end{equation}
and with that,
\begin{align}
\beta_{j,i} &= \eta_0(U_i,x_j) = \int_{U_i}\eta_0(du,x_j)\label{convergence:usual_approx_c},
\quad j=0,\ldots,2^{m}-1,\, i=0,\ldots,2^{k_m}, \\
\zeta_{1,i} &= \eta_1(U_i,e_l) = \int_{U_i}\eta_1(du,e_l),\quad \zeta_{2,i} = \eta_1(U_i,e_r) = \int_{U_i}\eta_1(du,e_r)
\label{convergence:usual_approx_s},\\
&\hspace{7.8cm} i=0,\ldots,2^{k_m}.\nonumber
\end{align}
The following two facts can easily be derived using the uniform continuity of $c_0$, $c_1$ as well as $f_1,\ldots,f_n$ (recall that 
these are continuous functions on a compact set), 
and the specific forms of the generators $A$ and $B$ as in seen
(\ref{introduction:notation-and-formalism:form-of-a}) and  (\ref{introduction:notation-and-formalism:form-of-b}), respectively.
\begin{lem}
\label{convergence:helper_lemma2}
 For $\epsilon>0$, there is a $\delta>0$ such that, uniformly in $x\in E$,
 \begin{equation*}
 \max\left\{\vert c_0(x,u) - c_0(x,v)\vert, \vert A f_1(x,u)-A f_1(x,v)\vert, \ldots, \vert A f_n(x,u)-A f_n(x,v)\vert \right\} < \epsilon
 \end{equation*}
 holds whenever $\vert u - v \vert<\delta $.
\end{lem}
\begin{lem}
\label{convergence:helper_lemma2_singular}
 For $\epsilon>0$, there is a $\delta>0$ such that for $s=e_l$ or $s=e_r$,
 \begin{equation*}
 \max\left\{\vert c_1(s,u) - c_1(s,v)\vert, \vert Bf_1(s,u)-Bf_1(s,v)\vert, \ldots, \vert Bf_n(s,u)-Bf_n(s,v)\vert \right\} < \epsilon
 \end{equation*}
 holds whenever $\vert u - v \vert<\delta $.
\end{lem}
The following two results ensure that we can approximate the cost criterion of a pair of measures $(\mu_{0}, \mu_{1})\in\mathscr{M}^l_n$ 
arbitrarily closely, and that the our approximate measures `almost' satisfy the linear constraints.
\begin{prop}
\label{convergence:cauchytrick_proposition2}
Consider a regular conditional probability $\eta_0$ and a probability density function $p$ stemming from a continuous occupation measure
$\mu_0$ such that $\mu_0(dx\times du) = \eta_0(du,x)p(x)dx$.
Let $g(x,u) = c_0(x,u)$ or $g(x,u) = A f_k(x,u)$ for any $k=1,2,\ldots,n$. For $\epsilon>0$, there exists an $m_0\in \mathbb{N}$ such that for all $m\geq m_0$,
\begin{equation}
\left\vert \int_E \int_U g(x,u) \eta_0(du,x)p(x)dx - \int_E \int_U g(x,u) \hat{\eta}_{0,m}(du,x)p(x)dx \right\vert <\epsilon,
\end{equation}
where $\hat{\eta}_{0,m}$ is of the form (\ref{approximation:mnk_kernel_c}), using the coefficients specified in 
(\ref{convergence:usual_approx_c}).
\end{prop}
\begin{proof}
Observe that given (\ref{approximation:mnk_kernel_c}),
\begin{align*}
\left\vert I \right\vert&\equiv\left\vert \int_E\int_U g(x,u) \eta_0(du,x)p(x)dx -\int_E\int_U g(x,u) \hat{\eta}_{0,m}(du,x)p(x)dx \right\vert\\
&= \left\vert \int_E\int_U g(x,u) \eta_0(du,x)p(x)dx -\int_E\left(\sum_{j=0}^{2^{m}-1}\sum_{i=0}^{2^{k_m}} \beta_{j,i}I_{E_j}(x)g(x,u_i)\right)p(x)dx \right\vert.\\
\end{align*}
By the definition of $\beta_{j,i}$ in (\ref{convergence:usual_approx_c}) and the triangle inequality it follows that
\begin{align*}
\left\vert I \right\vert &\leq \left\vert \int_E\int_U g(x,u) \eta_0(du,x)p(x)dx - \int_E\left(\sum_{i=0}^{2^{k_m}}g(x,u_i)\int_{U_i}\eta_0(du,x)\right)p(x)dx\right\vert\\
    &+ \left\vert \int_E\left(\sum_{i=0}^{2^{k_m}}g(x,u_i)\int_{U_i}\eta_0(du,x)-\sum_{j=0}^{2^{m}-1}\sum_{i=0}^{2^{k_m}}\int_{U_i}\eta_0(du,x_j)I_{E_j}(x)g(x,u_i)\right)p(x)dx\right\vert\\
    &\equiv \left\vert I_1 \right\vert+ \left\vert I_2 \right\vert.
    \end{align*}
Observe that
\begin{align*}
\left\vert I_1\right\vert = \left\vert \int_E \left( \sum_{i=0}^{2^{k_m}}\int_{U_i} \left(g(x,u)-g(x,u_i)\right) \eta_0(du,x)\right)p(x)dx
\right\vert.
\end{align*}
By Lemma \ref{convergence:helper_lemma2}, there is a $\delta>0$ such that for all $\vert u-v \vert<\delta$, we have that
$\vert g(x,u)-g(x,v)\vert<\frac{\epsilon}{2}$, uniformly in $x\in E$. Choose $m_1$ large enough such that for all $m\geq m_1$ it is true that
$\frac{1}{2^{k_m}}<\delta$. Then
\begin{equation}
\vert I_1\vert  < \left\vert \int_E \left( \sum_{i=0}^{2^{k_m}}\int_{U_i} \frac{\epsilon}{2} \eta_0(du,x)\right)p(x)dx\right\vert
= \frac{\epsilon}{2}.
\end{equation}
We now examine the term
\begin{equation}
\left\vert I_2\right\vert = \left\vert \int_E\sum_{i=0}^{2^{k_m}}g(x,u_i)
\left(\int_{U_i}\eta_0(du,x)-\sum_{j=0}^{2^{m}-1}I_{E_j}(x)\int_{U_i}\eta_0(du,x_j)\right)p(x)dx\right\vert.
\label{analysis:inft_bss_bcs:i2_definition}
\end{equation}
Regard $I_2$ as a sequence with two indices, say $I_2(a,b) \equiv I_2(a, k_b)\equiv I_2$, where $a$ and $b$ are two
discretization levels, with a slight abuse of notation.\\
Our first claim is that $I_2(a,b)$ is a Cauchy sequence in $b$. To see this, we analyze two successive elements of the
sequence. Consider
\begin{align*}
\left\vert I_2(a,b+1)-I_2(a,b)\right\vert &= \left\vert I_2(a, k_{b+1}) - I_2(a, k_{b}) \right\vert\\
&=\left\vert \int_E\left[ \sum_{i=0}^{2^{k_{b+1}}}g(x,\tilde{u}_i)\left(\eta_0(\tilde{U}_i,x)-\sum_{j=0}^{2^{a}-1}I_{E_j}(x)\eta_0(\tilde{U}_i,x_j)\right)\right.\right.\\
&\left.\left.\quad-\sum_{i=0}^{2^{k_{b}}}g(x,u_i)\left(\eta_0(U_i,x)-\sum_{j=0}^{2^{a}-1}I_{E_j}(x)\eta_0(U_i,x_j)\right)\right]p(x)dx\right\vert
\end{align*}
where $\tilde{U}_i$ and $\tilde{u_i}$ are used to indicate the partition of $U$ and the points of the discrete set in $U$ of the
discretization level $b+1$, defined analogously to (\ref{convergence:discrete-U-set-points}) and (\ref{convergence:discrete-U-set}).
Due to the additivity of measures, the two sums over $i$, if regarded as a Riemann-type approximation to an integral, only differ by a 
more accurate choice of the `rectangle height' $g(x,u_i)$ and $g(x,\tilde{u}_i)$ in the Riemann sum. To formalize this, for 
$i\in\{0,\ldots,2^{k_{b+1}}\}$ let $\pi(i)\in\{0,\ldots,2^{k_{b}}\} $
be the index such that $\tilde{U}_i \subset U_{\pi(i)}$. 
Observe that $\sum_{i=0}^{2^{k_{b+1}}}\left\vert  \eta_0(U_i,x)-\eta_0(U_i,x_j)\right\vert \leq 2$, independently of $x_j$, 
and thus independently of our choice of $E^{(a)}$. This is due to the fact
regular conditional probabilities are indeed probability measures with a full mass of $1$.
Then,
\begin{align*}
&\left\vert I_2(a,b+1)-I_2(a,b)\right\vert\\
&=\left\vert\int_E \left[\sum_{i=0}^{2^{k_{b+1}}}\left( g(x,\tilde{u}_i) - g(x,u_{\pi(i)})\right)\cdot
\left(\eta_0(\tilde{U}_i,x)-\sum_{j=0}^{2^{a}-1}I_{E_j}(x)\eta_0(\tilde{U}_i,x_j)\right)\right] p(x)dx\right\vert\\
&\leq\int_E \left[\sum_{i=0}^{2^{k_{b+1}}}\left\vert\ g(x,\tilde{u}_i) - g(x,u_{\pi(i)})\right\vert\cdot
\left\vert\eta_0(\tilde{U}_i,x)-\sum_{j=0}^{2^{a}-1}I_{E_j}(x)\eta_0(\tilde{U}_i,x_j)\right\vert\right] p(x)dx\\
&\leq K \left(\frac{1}{2}\right)^{b+1}\int_E \left[\sum_{i=0}^{2^{k_{b+1}}}\left\vert\eta_0(\tilde{U}_i,x)-
\sum_{j=0}^{2^{a}-1}I_{E_j}(x)\eta_0(\tilde{U}_i, x_j)\right\vert\right] p(x)dx\\
&\leq K \left(\frac{1}{2}\right)^{b+1}\cdot 2 =K \left(\frac{1}{2}\right)^{b}
\end{align*}
by the fact that 
$\left\vert\ g(x,\tilde{u}_i) - g(x,u_{\pi(i)})\right\vert$ is uniformly bounded by $K \left(\frac{1}{2}\right)^{b+1}$,
with $K=1$ if $g(x,u) = c_0(x,u)$, and $K=\max\{\|f_1\|_\mathscr{D}, \ldots, \|f_k\|_\mathscr{D}\}$ if 
$g(x,u) = Af_k(x,u)$ by our choice of $U^{\left(k_m\right)}$, compare (\ref{convergence:discrete-U-bounds}).\\
Now, for some $\vartheta>0$, choose $b$ large enough such that
$\sum_{j=b}^\infty \left(\frac{1}{2}\right)^j < \frac{\vartheta}{K}$. Then, for all 
$b_1>b_2\geq b$, we have
\begin{align}
\vert I_2(a,b_1)-I_2(a,b_2) \vert &= \left\vert\sum_{j=b_2}^{b_1-1}I_2(a,j+1)-I_2(a,j)\right\vert\nonumber\\
&\leq\sum_{j=b_2}^{b_1-1}\left\vert I_2(a,j+1)-I_2(a,j)\right\vert\nonumber\\
&\leq K\sum_{j=b_2}^{b_1-1} \left(\frac{1}{2}\right)^{j}<\vartheta,\label{convergence:bound-cauchy-increment}
\end{align}
revealing that $I_2$ is Cauchy in $b$, which is the same as saying it is Cauchy in $k_b$. The bound on the increment
of $I_2$ in $k_b$ (given by (\ref{convergence:bound-cauchy-increment})) is independent of $a$, so it does not depend on the choice of
$E^{(a)}$. Given this result,
choose $m_2\geq m_1$ such that for all $m\geq m_2$, we have that $k_m$ is large enough to make 
$\vert I_2(m,b_1)- I_2(m,b_2) \vert <\frac{\epsilon}{4}$ hold for all $b_1, b_2 \geq k_m$.
An application of the dominated convergence theorem together with the assumption that
$x\mapsto \eta(U_i,x)$ is continuous almost everywhere reveals that the choice of coefficients in (\ref{convergence:usual_approx_c}) 
ensures that for $k_m$ fixed and for each $i\in \{0,1,\ldots,2^{k_m}\}$, there is a $m^{(1,i)}$ large enough such that for all 
$m^{(i)}\geq m^{(1,i)}$, we have
\begin{align}
&\int_E \left\vert \eta_0(U_i,x)-\sum_{j=0}^{2^{m^{(i)}}-1}I_{E_j}(x)\eta_0(U_i,x_j)\right\vert p(x)dx \nonumber \\
&\qquad\qquad\qquad\qquad<\frac{\epsilon}{4\max\left\{\|c_0\|_\infty,\|Af_1\|_\infty,\ldots \|Af_n\|_\infty\right\}(2^{k_m+1})}
\label{convergence:approximation_of_control}
\end{align}
Set $\tilde{m}=\max\left\{\max_{i=0,\ldots,2^{k_m}}m^{(1,i)},m_2\right\}$. Then,
\begin{align*}
I_2(\tilde{m},k_m) &\leq \int_E\|g\|_\infty \sum_{i=0}^{2^{k_m}}\left\vert \eta_0(U_i,x)-\sum_{j=0}^{2^{\tilde{m}}-1}I_{E_j}(x)\eta_0(U_i,x_j)\right\vert p(x)dx\\
&\leq\|g\|_\infty \sum_{i=0}^{2^{k_m}} \int_E\left\vert \eta_0(U_i,x)-\sum_{j=0}^{2^{\tilde{m}}-1}I_{E_j}(x)\eta_0(U_i,x_j)\right\vert p(x)dx\\
&\leq\frac{\epsilon}{4}.
\end{align*}
Note that $I_2(\tilde{m},k_m)$ is decreasing in $\tilde{m}$, which again is revealed using the dominated convergence theorem.
Also,
for $l\geq k_m$, we have
\begin{align*}
\left\vert I_2(\tilde{m},l) \right\vert &\leq \left\vert I_2(\tilde{m},l) - I_2(\tilde{m},k_m) \right\vert +
\left\vert I_2(\tilde{m},k_m) \right\vert\\
&< \frac{\epsilon}{4} + \frac{\epsilon}{4} = \frac{\epsilon}{2}
\end{align*}
which means that $I_2 \equiv I_2(\tilde{m},k_m)$ does not exceed $\frac{\epsilon}{2}$ when $\tilde{m}$ or $k_m$ increase.
Choose $m_0= \max\{\tilde{m},m_2\}$. Then, for all 
$m\geq m_0$, we have that $I_2<\frac{\epsilon}{2}$.
\end{proof}
\begin{prop}
\label{convergence:cauchytrick_proposition2_singular}
Consider a singular occupation measure $\mu_1$ that decomposes into $\mu_1(dx\times du) = \eta_1(du,x)\mu_{1,E}(dx)$. 
Let $g(x,u) = c_1(x,u)$ or $g(x,u) = Bf_k(x,u)$ for any $k=1,2,\ldots,n$. For $\epsilon>0$, there exists an $m_0\in \mathbb{N}$ such that for all
$m\geq m_0$,
\begin{equation}
\left\vert \int_E \int_U g(x,u) \eta_1(du,x)\mu_{1,E}(dx) - \int_E \int_U g(x,u) \hat{\eta}_{1,m}(du,x)\mu_{1,E}(dx) \right\vert <\epsilon,
\end{equation}
where $\hat{\eta}_{1,m}$ is of the form (\ref{approximation:mnk_kernel_s}), using the coefficients specified in 
(\ref{convergence:usual_approx_s}).
\end{prop}
\begin{proof}
 We only have to show that  
 \begin{equation}
 \left\vert \int_U g(s,u) \eta_1(du,s) - g(s,u) \hat{\eta}_{1,m}(du,s) \right\vert < \frac{\epsilon}{\mu_{1,E}(E)}
 \end{equation}
 uniformly for $s=e_l$ or $s=e_r$, since $\mu_{1,E}$ only puts mass on these two points. By (\ref{convergence:usual_approx_s}),
 \begin{align*}
 &\left\vert \int_U g(s,u) \eta_1(du,s) - \int_Ug(s,u) \hat{\eta}_{1,m}(du,s) \right\vert \\
 &\qquad\qquad\qquad=\left\vert \sum_{i=0}^{2^{k_m}}\int_{U_i} g(s,u) \eta_1(du,s) -\int_{U_i} g(s,u) \hat{\eta}_{1,m}(du,s)\right\vert\\
 &\qquad\qquad\qquad=\left\vert \sum_{i=0}^{2^{k_m}}\int_{U_i} g(s,u) \eta_1(du,s) - g(s,u_i) \zeta_{j,i}\right\vert\\ 
 &\qquad\qquad\qquad=\left\vert \sum_{i=0}^{2^{k_m}}\int_{U_i} g(s,u) \eta_1(du,s) - g(s,u_i) \int_{U_i}\eta_1(du,s)\right\vert\\ 
 &\qquad\qquad\qquad\leq \sum_{i=0}^{2^{k_m}}\int_{U_i} \vert g(s,u) - g(s,u_i) \vert\,\eta_1(du,s)
 \end{align*}
holds. According to Lemma \ref{convergence:helper_lemma2_singular}, there is a $\delta>0$ such that 
$\vert g(s,u) - g(s,v) \vert<\frac{\epsilon}{\mu_{1,E}(E)}$ whenever $\vert u-v \vert<\delta$. Hence it suffices to choose $m_0$ large enough such
that for all $m\geq m_0$ it is true that $\frac{1}{2^{k_m}}<\delta$ to ensure that $\vert u - u_i \vert<\delta$.
\end{proof}
Proposition \ref{convergence:cauchytrick_proposition2} and Proposition \ref{convergence:cauchytrick_proposition2_singular} have established the 
approximation properties of the coefficient choices made by (\ref{convergence:usual_approx_c}) and (\ref{convergence:usual_approx_s}),
without paying any respect to the constraints
defining $\mathscr{M}^l_{n,m}$. We proceed to link these approximate
controls to measures that actually fulfill those constraints. To do so, we need to be able to quantify how far away a given approximation is 
from satisfying the constraints. This motivates the following definitions.
\begin{defn}
Let $\eta_0$ be any relaxed control. For $n,m\in \mathbb{N}$, define the constraint matrix 
$C^{(m)}\in \mathbb{R}^{n+1,2^{m}}$ by 
\begin{align*}
C^{(m)}_{k,j} &=& \int_E \int_U A f_k(x,u)\eta_0(du,x)p_j(x)dx, \quad k=1,2,\ldots,n,\quad j=0,1,\ldots,2^{m}-1\\
C^{(m)}_{n+1,j} &=& \int_E p_j(x)dx \qquad \qquad \quad \,\,\quad \quad\qquad \qquad \qquad \qquad \quad\ j=0,1,\ldots, 2^{m}-1.
\end{align*}
\end{defn}
\begin{defn}
 For $m\in \mathbb{N}$ and $\hat{\eta}_{0,m}$ given by (\ref{convergence:usual_approx_c}), $\hat{\eta}_{1,m}$ given by
 (\ref{convergence:usual_approx_s}) and some $\tilde{p}$ in the span of
 $\{p_0,p_1,\ldots,p_{2^{m}-1}\}$, with $\mu_{1,E}$ being the state space marginal of the previously fixed measure $\mu_{1}$, 
 the constraint error $d^{(m)}(\tilde{p})\in \mathbb{R}^{n+1}$ is defined for $k=1,\ldots,n$ by
 \begin{align*}
 d_k^{(m)}(\tilde{p}) &= Rf_k -\int_E \int_U A f_k(x,u)\hat{\eta}_{0,m}(du,x) \tilde{p}(x)dx\\
 &\qquad \qquad \qquad \qquad - \int_E \int_U Bf_k(x,u)\hat{\eta}_{1,m}(du,x)\mu_{1,E}(dx),
 \end{align*}
 and for $k={n+1}$ by $d^{(m)}_{n+1}=1-\int_E \tilde{p}(x)(dx)$.
\end{defn}
 One can increase $m$ to the point that the constraint matrix $C^{(m)}$ has full rank. This will help us find adjustments to the coefficients of $\tilde{p}$
 which will let the constraint error vanish. The following results show
 that we can attain an arbitrarily small constraint error using the proposed approximation. Their proofs are rather technical, and are
 given in \Cref{appendix_a}.
 \begin{lem}
\label{analysis:inft_bss_bcs:suitable_approximation}
 Let $p$ be a probability density function with $\lambda\left(\left\{x : p(x)=0\right\} \right) = 0$. Then, for any $\epsilon>0$ and 
 $D_1>0$,
 there exists an $\hat{\epsilon}_1<\epsilon$ and an $m_0$ such that for all $m\geq m_0$, there is a $\tilde{p}_{m}$ in the span of
 $\{p_0,p_,\ldots, p_{2^{m}-1}\}$ with $\|p-\tilde{p}_{m}\|_{L^1(E)}<\frac{\hat{\epsilon}_1}{D_1}$ and $\tilde{p}_{m}\geq \hat{\epsilon}_1$ 
 on $E$.
\end{lem}
  \begin{lem}
\label{analysis:inft_bss_bcs:bound_on_const_error}
 Consider a pair of measures $(\mu_0,\mu_1)\in\mathscr{M}_{n,\infty}$, and let $\mu_0(dx\times du) = \eta_0(du,x)\,p(x)\,dx$ as well as
 $\mu_1(dx\times du) = \eta_1(du,x)\,\mu_{1,E}(dx)$.
 Let \\$\bar{A} = \max_{k=1,\ldots,n} \|Af_k\|_\infty$. For $\delta>0$ and $D_2\geq 1$, there exists an 
 $\hat{\epsilon}_2\leq\delta$ and an
 $m_0\in \mathbb{N}$ such that for all $m\geq m_0$, there is a function $\tilde{p}_{m}$ in the span of $\{p_0,p_1,\ldots,p_{2^m-1}\}$ with 
 $\|d^{(m)}(\tilde{p}_{m})\|_\infty<\hat{\epsilon}_2$, where $d^{(m)}(\tilde{p}_{m})$ is the constraint error using
 the approximations $\hat{\eta}_{0,m}(du,x)$ and $\hat{\eta}_{1,m}(du,x)$ of the given controls $\eta_0$ and $\eta_1$ defined by the 
 coefficients given in (\ref{convergence:usual_approx_c}) and (\ref{convergence:usual_approx_s}). In particular,
 $\|p-\tilde{p}_{m}\|_{L^1(E)}<\frac{\hat{\epsilon}_2}{\rule{0pt}{9pt}3\max\{\bar{A},1\}}$ as well as 
 $\tilde{p}_{m}\geq D_2\cdot \hat{\epsilon}_2$ holds.
\end{lem}
Next, we establish that we can find a `correction' term $y$ for the coefficients of $\tilde{p}$, which can be used to define a measure that
satisfies the constraints, with a 
maximum norm that does not exceed a given bound $\epsilon$. For the proof of the following statement, we again refer to \Cref{appendix_a}.
\begin{lem}
Consider a pair of measures $(\mu_0,\mu_1)\in\mathscr{M}_{n,\infty}$, and let $\mu_0(dx\times du) = \eta_0(du,x)\,p(x)\,dx$ as well as
 $\mu_1(dx\times du) = \eta_1(du,x)\,\mu_{1,E}(dx)$.
 For any $\vartheta>0$, there is a $\hat{\vartheta}<\vartheta$ and an $m_0\in \mathbb{N}$ such that for all $m\geq m_0$,
 there is a $\tilde{p}_{m}\in \Span\{p_0,p_1,\ldots,p_{2^{m}-1} \}$ such that the equation
 $C^{(m)} y = -d^{(m)}(\tilde{p}_{m})$ has a solution $\tilde{y}$ with $\|\tilde{y}\|_\infty\leq\hat{\vartheta}$, 
 In particular, $\tilde{p}_{m}\geq\hat{\vartheta}$ and $\|p-\tilde{p}_{m}\|_{L^1(E)}<\hat{\vartheta}$ hold.
 \label{analysis:inft_bss_bcs:bound_on_correction}
\end{lem}
At this point, we have gathered the results to prove Proposition \ref{convergence:second_part_big_proposition}, stated again for the sake 
of presentation.
\setcounter{thm}{1}
\begin{prop} 
For $(\mu_0,\mu_1) \in \mathscr{M}^l_n$ and each $\epsilon>0$, there is an $m_0$ such that for all $m\geq m_0$ there exists a 
$(\hat{\mu}_{0,m}, \hat{\mu}_{1,m}) \in \mathscr{M}^l_{n,m}$,
with
\begin{equation*}
\left\vert J(\mu_0,\mu_1) - J(\hat{\mu}_{0,m}, \hat{\mu}_{1,m})\right \vert <\epsilon.
\end{equation*}
\end{prop}
\begin{proof} Fix $\epsilon>0$.  For $(\mu_0,\mu_1)\in\mathscr{M}^l_n$, let $\mu_{0,E}$ be the state space marginal of $\mu_0$ 
and let
$\eta_0$ be the regular conditional probability such that $\mu_0(dx\times du)=\eta_0(du,x)\mu_{0,E}(dx)$. Similarly, let 
$\mu_{1}(dx\times du) = \eta_1(du,x)\mu_{1,E}(dx)$. Define $\hat{\eta}_{0,m}$ and 
$\hat{\eta}_{1,m}$ using the coefficients given by (\ref{convergence:usual_approx_c}) and (\ref{convergence:usual_approx_s}), respectively. First, by 
Proposition \ref{convergence:cauchytrick_proposition2_singular}, we have that there is an $m_1$ such that $\forall\,m\geq m_1$,
\begin{equation}
\left\vert\int_E \int_U c_1(x,u)\eta_1(du,x) \mu_{1,E}(dx) - \int_E \int_U c_1(x,u)\hat{\eta}_{1,m}(du,x) \hat{\mu}_{1,E}(dx)\right\vert 
< \frac{\epsilon}{2}.
\end{equation}
Now, we consider the approximation of the cost accrued by $c_0$. We will show that
\begin{equation}
\left\vert \int_E\int_U c_0(x,u) \eta_0(du,x)p(x)dx - \int_E\int_U c_0(x,u) \hat{\eta}_{0,m}(du,x)\hat{p}_{m}(x)dx \right \vert <\frac{\epsilon}{2},
\end{equation}
for the given choice of $\hat{\eta}_{0,m}$ and a choice of $\hat{p}_{m}$ to be identified. 
This will be done by a successive application of the triangle inequality. First, observe that
\begin{align*}
&\left\vert \int_E\int_U c_0(x,u) \eta_0(du,x)p(x)dx - \int_E\int_U c_0(x,u) \hat{\eta}_{0,m}(du,x)\hat{p}_{m}(x)dx \right\vert\\
&\quad \leq\quad \left\vert \int_E\int_U c_0(x,u) \eta_0(du,x)p(x)dx - \int_E\int_U c_0(x,u) \hat{\eta}_{0,m}(du,x)p(x)dx \right\vert\\
&\qquad+ \left\vert \int_E\int_U c_0(x,u) \hat{\eta}_{0,m}(du,x)p(x)dx -\int_E\int_U c_0(x,u) \hat{\eta}_{0,m}(du,x)\hat{p}_{m}(x)dx \right\vert\\
&\quad\equiv \left\vert I_1 \right\vert + \left\vert I_2 \right\vert.
\end{align*}
Now set 
\begin{equation}
\vartheta = \min\left\{\frac{\epsilon}{8\|c_0\|_\infty\max\{1,(e_r-e_l)\}},\frac{3 \epsilon \max\{\bar{A},1\}}{8\|c_0\|_\infty}\right\}. 
\end{equation}
By Lemma \ref{analysis:inft_bss_bcs:bound_on_correction} we can choose an $m_2\geq m_1$ such that
for all $m\geq m_2$, there is a function 
$\tilde{p}_{m}=\sum_{i=0}^{2^{m}-1}\tilde{\gamma}_ip_i \in \Span\{p_0,p_1,\ldots,p_{2^{m}-1} \}$ that allows for a solution 
$\tilde{y}$ to $C^{(m)}y=-d^{(m)}(\tilde{p}_{m})$ with $\|\tilde{y}\|_\infty\leq\hat{\vartheta}<\vartheta$ for some $\hat{\vartheta}<\vartheta$,
but $\tilde{p}_{m}\geq\hat{\vartheta}$. This
$m_2$ is also large enough to approximate $p$ by $\tilde{p}_{m}$ with an accuracy of 
$\frac{\epsilon}{8\|c_0\|_\infty}$ for all $m\geq m_2$. 
Define new coefficients $\gamma_i = \tilde{\gamma}_i - \tilde{y}_i$ and set 
$\hat{p}_{m}=\sum_{i=0}^{2^{m}-1}\gamma_ip_i$. Then, for all $k=1,2,\ldots,n$,
\begin{align*}
d^{(m)}_k(\hat{p}_{m}) &= Rf_k - \left(C^{(m)}\gamma \right)_k - \int_E \int_U Bf_k(x,u)\hat{\eta}_{1,m}\mu_{1,E}(dx)\\
&= Rf_k - \left(C^{(m)}(\tilde{\gamma} - \tilde{y}) \right)_k - \int_E \int_U Bf_k(x,u)\hat{\eta}_{1,m}\mu_{1,E}(dx)\\
&= Rf_k - \left(C^{(m)}\tilde{\gamma}\right)_k - \int_E \int_U Bf_k(x,u)\hat{\eta}_{1,m}\mu_{1,E}(dx) + \left(C^{(m)}\tilde{y}\right)_k\\
&= d^{(m)}_k(\tilde{p}_{m}) - d^{(m)}_k(\tilde{p}_{m}) =0 
\end{align*}
and
\begin{equation*}
d^{(m)}_{n+1}(\hat{p}_{m}) =
1-(C^{(m)}\tilde{\gamma})_{n+1} + (C^{(m)}\tilde{y})_{n+1}=d^{(m)}_{n+1}(\tilde{p}_{m})-d^{(m)}_{n+1}(\tilde{p}_{m}) = 0
\end{equation*}
which shows that $\hat{p}_{m}$ fulfills
the constraints. But also, $\hat{p}_{m}\geq 0$. So\\
$(\hat{\eta}_{0,m}(du,x)\hat{p}_{m}(x)dx,\hat{\eta}_{0,m}(du,x)\mu_{1,E}(dx)\in \mathscr{M}^l_{n,m}$. Furthermore,
\begin{equation*}
\|p-\hat{p}_{m}\|_{L^1(E)} \leq \|p-\tilde{p}_m\|_{L^1(E)} + \|\tilde{p}_m-\hat{p}_m\|_{L^1(E)} \leq \frac{\epsilon}{8\|c_0\|_\infty} + \frac{\epsilon}{8\|c_0\|_\infty} = \frac{\epsilon}{4\|c_0\|_\infty}.
\end{equation*}
This shows that 
\begin{align*}
\left\vert I_2 \right\vert \quad\leq \int_E \|c_0\|_\infty\int_U \hat{\eta}_{0,m}(du,x) \vert p(x)-\hat{p}_{m}(x)\vert dx <\|c_0\|_\infty\|p-\hat{p}_{m} \|_{L^1}
<\frac{\epsilon}{4},
\end{align*}
since $\int_U \hat{\eta}_{0,m}(du,x)=1$. Turning to $\left\vert I_1 \right\vert$, we have seen in 
Proposition \ref{convergence:cauchytrick_proposition2}
that there is an $m_3\geq m_2$ such that for all $m\geq m_3$, $\vert I_1 \vert < \frac{\epsilon}{4}$.
To sum up,
\begin{align*}
&\left\vert \int_E\int_U c_0(x,u) \eta_0(du,x)p(x)dx - \int_E\int_U c_0(x,u) \hat{\eta}_{0,m}(du,x)\hat{p}_{m}(x)dx \right\vert\\
&\quad\leq I_1 + I_2\leq \frac{\epsilon}{4} + \frac{\epsilon}{4} = \frac{\epsilon}{2}
\end{align*}
But this gives us the assertion, and finishes the proof, setting $m_0=m_3$.
\end{proof}
\noindent So far, we have investigated the approximation of arbitrary measures. Proposition \ref{convergence:second_part_big_proposition} is
instrumental in proving the next important result, Corollary \ref{convergence:second_part_corollary}, which analyzes how optimal solutions in 
$\mathscr{M}^l_{n,m}$ relate to $\epsilon$-optimal solution in $\mathscr{M}_n^l$.
The following lemma will as well be needed. Its proof is similar to an argument used in (\ref{convergence:proposition1}), and thus is omitted.
\setcounter{thm}{14}
\begin{lem}
\label{convergence:closure_lemma}
Let $\{\mu_{0,n,m}, \mu_{1,n,m}\}$ be a sequence of measures such that for each $m$,
$(\mu_{0,n,m}, \mu_{1,n,m})\in \mathscr{M}^l_{n,m}$. Assume that
$\mu_{0,n,m} \Rightarrow \hat{\mu}_{0,n}$ and $\mu_{1,n,m}\Rightarrow \hat{\mu}_{1,n}$ as $m \rightarrow \infty$.
Then, $(\hat{\mu}_{0,n}, \hat{\mu}_{1,n}) \in \mathscr{M}^l_n$.
\end{lem}
\setcounter{thm}{2}
Now we can prove Corollary \ref{convergence:second_part_corollary} from \Cref{convergence-main-results}.
\begin{crl}
 For each $m\in \mathbb{N}$, assume that $(\mu^*_{0,n,m}, \mu^*_{1,n,m})\in \mathscr{M}^l_{n,m}$ and that for $m\in \mathbb{M}$, 
$(\mu^*_{0,n,m}, \mu^*_{1,n,m})$ is an optimal solution to the $l$-bounded, $(n,m)$-dimensional linear program.
Then, the sequence of numbers $\{J(\mu^*_{0,n,m}, \mu^*_{1,n,m})\}_{m\in\mathbb{N}}$\\
converges to $J^*_n= \inf_{(\mu_{0,n}, \mu_{1,n})\in \mathscr{M}^l_n} J(\mu_{0,n}, \mu_{1,n})$ as $m\rightarrow \infty$.
\end{crl}
\begin{proof} 
First, observe that if $\mu^*_{0,n,m} \Rightarrow \hat{\mu}^*_{0,n}$ and 
$\mu^*_{1,n,m}\Rightarrow \hat{\mu}^*_{1,n}$ as $m \rightarrow \infty$ for some 
$(\hat{\mu}^*_{0,n}, \hat{\mu}^*_{1,n})\in \mathscr{M}^l_n$, it follows that $J(\hat{\mu}^*_{0,n}, \hat{\mu}^*_{1,n}) = J^*_n$.
The proof of this claim is as follows. Assume the opposite. Then, there is a pair of measures 
$(\hat{\mu}_{0,n}, \hat{\mu}_{1,n})\in \mathscr{M}^l_n$ with 
$J(\hat{\mu}^*_{0,n}, \hat{\mu}^*_{1,n})-J(\hat{\mu}_{0,n}, \hat{\mu}_{1,n})>\epsilon$ for some $\epsilon >0$.
Select $m_0$ large enough such that for all $m\geq m_0$, 
$\vert J(\hat{\mu}^*_{0,n}, \hat{\mu}^*_{1,n}) - J(\mu^*_{0,n,m}, \mu^*_{1,n,m})\vert <\frac{\epsilon}{2}$ and hence
$J(\mu^*_{0,n,m}, \mu^*_{1,n,m})-J(\hat{\mu}_{0,n}, \hat{\mu}_{1,n})>\frac{\epsilon}{2}$ for all $m\geq m_0$.
By Proposition \ref{convergence:second_part_big_proposition}, select $m\geq m_0$ large enough  that there is a a pair of measures
$(\hat{\mu}_{0,n,m}, \hat{\mu}_{1,n,m})\in \mathscr{M}^l_{n,m}$ with 
$\vert J(\hat{\mu}_{0,n}, \hat{\mu}_{1,n}) - J(\hat{\mu}_{0,n,m}, \hat{\mu}_{1,n,m})\vert <\frac{\epsilon}{2}$. But then,
$J(\hat{\mu}_{0,n,m}, \hat{\mu}_{1,n,m})<J(\mu^*_{0,n,m}, \mu^*_{1,n,m})$, contradicting that $(\mu^*_{0,n,m}, \mu^*_{1,n,m})$ is the 
optimal solution in $\mathscr{M}^l_{n,m}$.\\
Also, $\{J(\mu^*_{0,n,m}, \mu^*_{1,n,m})\}_{m\in\mathbb{N}}$ is a decreasing sequence which is bounded from below, 
so it converges. As $\{\mu^*_{0,n,m}\}_{m\in\mathbb{N}}$ and $\{\mu^*_{1,n,m}\}_{m\in\mathbb{N}}$ are sequences of measures over a compact space, they are tight,
and the full mass of $\{\mu^*_{1,n,m}\}_{m\in\mathbb{N}}$ is uniformly bounded by $l$. So there is a convergent subsequence 
$\{(\mu^*_{0,n,{m_1}}, \mu^*_{1,n,m_1})\}_{m_1\in\mathbb{N}}$ with
\begin{equation}
\mu^*_{0,n,{m_1}} \Rightarrow \hat{\mu}^*_{0,n} \qquad\mbox{and}\qquad \mu^*_{1,n,{m_1}} \Rightarrow \hat{\mu}^*_{1,n}
\end{equation}
for some $(\hat{\mu}^*_{0,n},\hat{\mu}^*_{1,n})\in\mathscr{M}_n$. By the first part of this proof and because $c_0$ and $c_1$ are
bounded and uniformly continuous,
\begin{align*}
J^*_n = J(\hat{\mu}_{0,n}^*,\hat{\mu}_{1,n}^*) &= \int c_0 d\hat{\mu}^*_{0,n} + \int c_1 d\hat{\mu}^*_{1,n}\\
&= \lim_{m_1\rightarrow \infty}\left( \int c_0 d\hat{\mu}^*_{0,n,{m_1}} + \int c_1 d\hat{\mu}^*_{1,n,{m_1}}\right)\\
&= \lim_{m_1\rightarrow \infty} J(\mu^*_{0,n,{m_1}}, \mu^*_{1,n,{m_1}}),
\end{align*}
but $\{J(\mu^*_{0,n,m}, \mu^*_{1,n,m})\}_{m\in\mathbb{N}}$ converges, and any subsequence has to converge to its very limit. So,
\begin{equation*}
\lim_{m\rightarrow \infty} J(\mu^*_{0,n,m}, \mu^*_{1,n,m}) = J^*_n
\end{equation*}
\end{proof}

\section{Examples}
\label{example-section}
\subsection{Modified Bounded Follower}
\label{mbf-subsection}
\label{example}
Consider a stochastic control problem with state space $E=[0,1]$ such that the process is governed by the SDE
\begin{equation*}
dX_t =  u(X_t) dt + \sigma dW_t +d\xi_t,\quad X_0 = x_0
\end{equation*}
in which $u(x)\in U=[-1,1]$, and $\xi$ is a process that captures the singular behavior of $X$. The latter is given by a 
reflection to the right at $\{0\}$ and a jump from $\{1\}$ to $\{0\}$.
We use the relaxed martingale formulation, compare Definition \ref{introduction:def-relaxed-mgp}, and retain the coefficient functions $b(x,u) = u$ and
$\sigma(x,u)\equiv \sigma$. We adopt the long-term average cost criterion, with cost functions $c_0(x,u) = x^2$, $c_1(e_r,u) \equiv c_1$ at the
right endpoint for some
$c_1\in \mathbb{R}_+$ and $c_1(e_l,u)=0$ at the left endpoint. This problem is known as the modified bounded follower in the literature. 
According to \cite{helmes-stockbridge-oc}, the optimal control for this problem is a degenerate relaxed control $\eta_0$ with 
$\eta_0(\{u_a(x)\},x)=1$, where $u_a$ is of the form
\begin{equation*}
u_a(x) = \left\{ \begin{array}{ll} -1 & x< a \\ +1 & x\geq a.\\ \end{array}\right.
\end{equation*}
The `switching point' $a$ depends on the coefficient  and cost functions. Furthermore, the state space marginal 
$\mu_{0,E}$ under the optimal control has the density 
\begin{equation*}
p_a(x) = \frac{\int_x^1\exp\left(\int_x^y-\frac{2}{\sigma^2}u_a(z)\,dz\right)\,dy}{\int_0^1\int_x^1\exp\left(\int_x^y-\frac{2}{\sigma^2}u_a(z)\,dz\right)\,dy\,dx}.
\end{equation*}
We will compare the performance of the proposed numerical method against this analytic solution. \Cref{example:configuration} shows the configuration of 
the problem, along with the optimal switching point $a$ under this configuration and the weights of the occupation measure $\mu_1$, capturing the singular 
behavior on the left boundary $\{0\}$ and on the right boundary $\{1\}$, denoted $w_1$ and $w_2$, respectively. It also shows the value of the
cost criterion $J^*$ under the optimal control.
\begin{table}[ht]
\caption{Configuration and analytic solution, modified bounded follower}
\label{example:configuration}
 \centering
 \begin{tabular}{ccccccc}
 \hline
 \hline
  $x_0$ & $\sigma$ & $c_1$ & $a$ & $w_1$ & $w_2$ & $J^*$\\
  $0.1$ & $\sqrt{2}$ & $0.01$ & $0.7512$ & $2.4659$ & $1.5555$ & $0.1540$\\
  \hline
  \hline
 \end{tabular}
\end{table}
\begin{figure}[t]
    \centering
    \begin{minipage}{0.46\textwidth}
        \centering
	\includegraphics[width=1\textwidth]{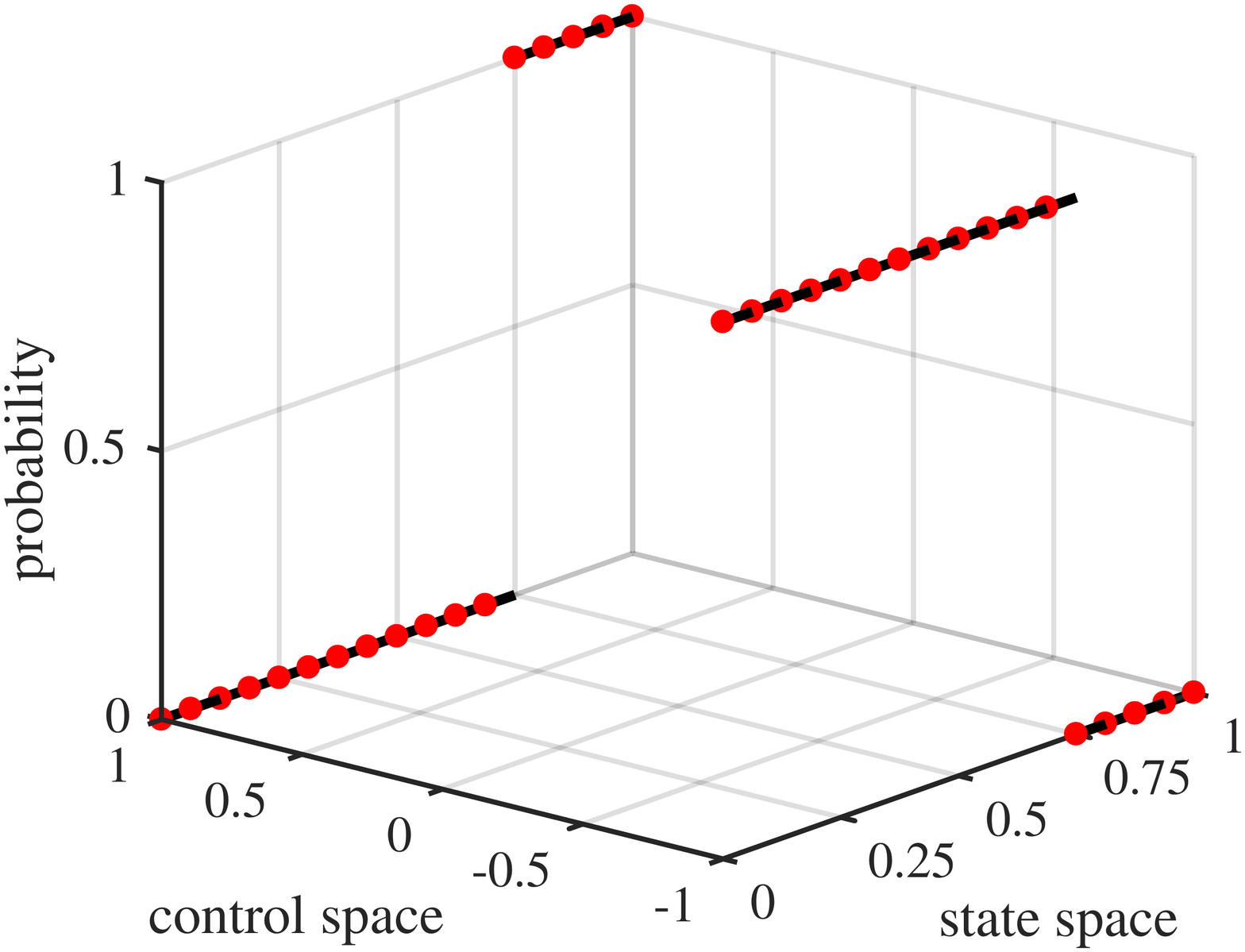}
        \caption{Computed optimal relaxed control, coarse grid, modified bounded follower}
        \label{example:plot_mbf_relaxed_control_coarse}
    \end{minipage}\hfill
    \begin{minipage}{0.46\textwidth}
        \centering
	\includegraphics[width=1\textwidth]{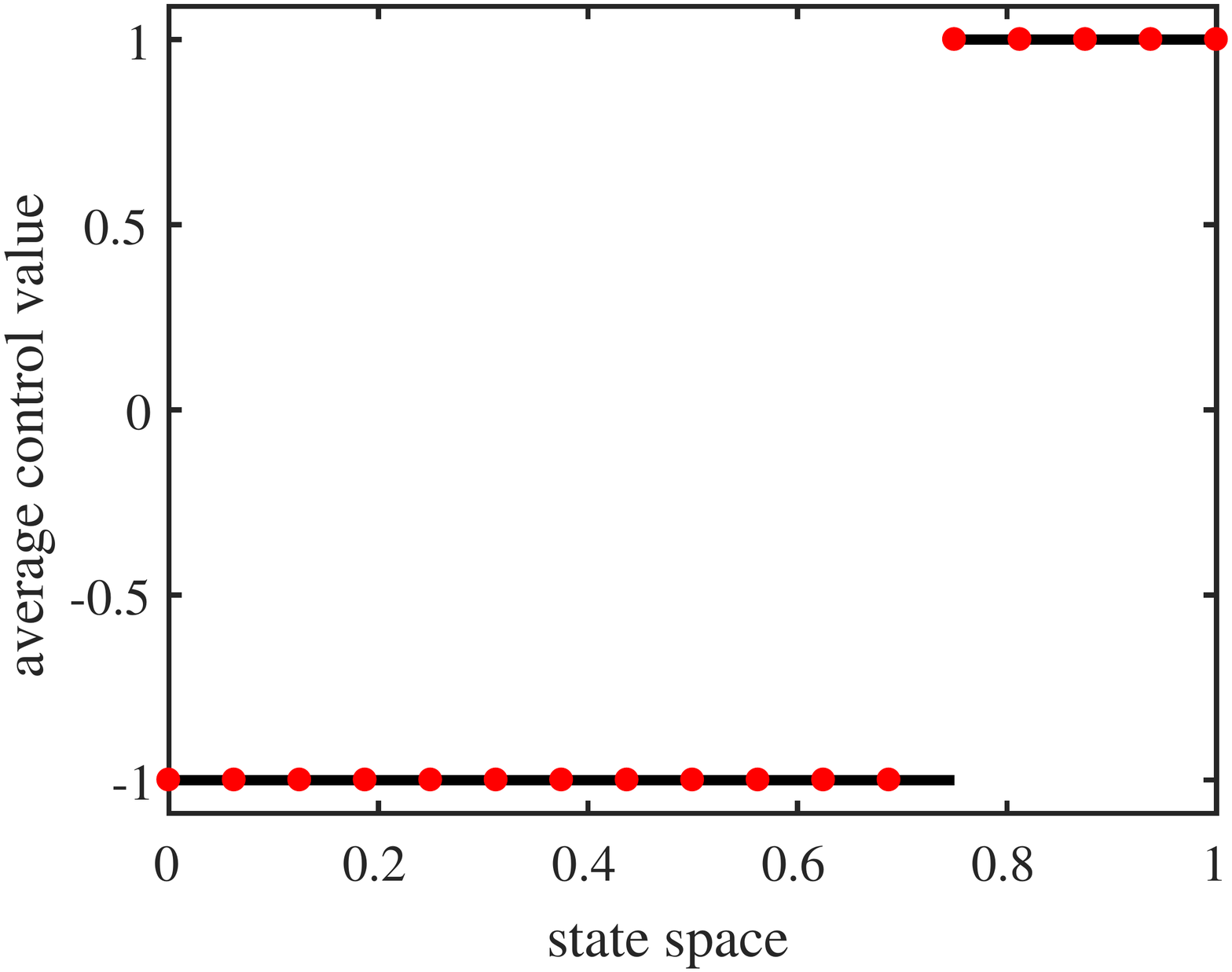}
        \caption{Computed optimal average control, coarse grid, modified bounded follower}
        \label{example:plot_mbf_average_control_coarse}
    \end{minipage}
\end{figure}
\begin{table}[ht]
\caption{Results (1), modified bounded follower}
\label{example:results1}
 \centering
 \begin{tabular}{|l|l||l||l|l|l||l|}
    \hline
    $n$ & $m$ & $T$ & $J^*_{n,m}$ & $e_a$ & $e_r$ & $e_{L^1}$\\
    \hline
    $3$ & $3$ & $0.0089$ & $0.15400$ & $9.798\cdot 10^{-6}$ & $6.363\cdot 10^{-5}$ & $9.082\cdot 10^{-2}$\\
    $4$ & $4$ & $0.0091$ & $0.15399$ & $1.199\cdot 10^{-6}$ & $7.789\cdot 10^{-6}$ & $4.563\cdot 10^{-2}$\\
    $5$ & $5$ & $0.0101$ & $0.15399$ & $5.447\cdot 10^{-7}$ & $3.537\cdot 10^{-6}$ & $2.287\cdot 10^{-2}$\\
    $6$ & $6$ & $0.0124$ & $0.15399$ & $4.809\cdot 10^{-7}$ & $3.123\cdot 10^{-6}$ & $1.145\cdot 10^{-2}$\\
    $7$ & $7$ & $0.0173$ & $0.15399$ & $4.713\cdot 10^{-7}$ & $3.061\cdot 10^{-6}$ & $5.73\cdot 10^{-3}$\\
    $8$ & $8$ & $0.0338$ & $0.15399$ & $4.694\cdot 10^{-7}$ & $3.048\cdot 10^{-6}$ & $2.866\cdot 10^{-3}$\\
    $9$ & $9$ & $0.0844$ & $0.15399$ & $3.306\cdot 10^{-7}$ & $2.147\cdot 10^{-6}$ & $1.451\cdot 10^{-3}$\\
   $10$ & $10$ & $0.2754$ & $0.15399$ & $2.655\cdot 10^{-7}$ & $1.724\cdot 10^{-6}$ & $7.258\cdot 10^{-4}$\\
   $11$ & $11$ & $0.9586$ & $0.15399$ & $3.550\cdot 10^{-7}$ & $2.305\cdot 10^{-6}$ & $4.509\cdot 10^{-4}$\\
\hline
 \end{tabular}
\end{table}
\Cref{example:results1} and \Cref{example:results2} show the results and
performance measures for various discretization levels $n$ and $m$.
To achieve higher accuracy, we added another mesh point for the choice of basis functions for $p$ by cutting the interval in the middle of the
state space in half.
As the cost function does not depend on $u$ we expect the optimal solution to be a bang-bang control. Hence it suffices to choose $k_m=0$, which
means the optimization has to choose between two possible control values $\{-1,1\}$.\\
\Cref{example:results1} shows the result for the approximate cost criterion $J^*_{n,m}$. The column $e_a$ refers to the absolute
error between $J^*$ and $J^*_{n,m}$, the column $e_r$ to the relative error  between $J^*$ and $J^*_{n,m}$ and the column $e_{L^1}$ to the $L^1(E)$-distance between $\hat{p}_m$ and $p$.
$T$ is the execution time, which is an average time taken from $1000$ repetitions of the same optimization run.
In \Cref{example:results2}, $\hat{w}_1$ and $\hat{w}_2$ refer to the 
approximate values for $w_1$ and $w_2$.  The discretization levels $m$ for \cref{example:results2} are the same
as in \cref{example:results1} for respective $n$, and as before, $e_a$ refers to the absolute error and $e_r$ refers to the relative
error of these quantities.
\begin{table}[ht]
\caption{Results (2), modified bounded follower}
\label{example:results2}
 \centering
 \begin{tabular}{|l||l|l|l||l|l|l|}
   \hline
    $n$ & $\hat{w}_1$ & $e_a$ & $e_r$ & $\hat{w}_2$ & $e_a$ & $e_r$\\
    \hline
    $3$ & $2.4667$ & $7.887\cdot 10^{-4}$ & $3.199\cdot 10^{-4}$ & $1.5577$ & $2.157\cdot 10^{-3}$ & $1.387\cdot 10^{-3}$\\
    $4$ & $2.4661$ & $1.984\cdot 10^{-4}$ & $8.047\cdot 10^{-5}$ & $1.5560$ & $5.440\cdot 10^{-4}$ & $3.480\cdot 10^{-4}$\\
    $5$ & $2.4659$ & $4.969\cdot 10^{-5}$ & $2.015\cdot 10^{-5}$ & $1.5556$ & $1.355\cdot 10^{-4}$ & $8.700\cdot 10^{-5}$\\
    $6$ & $2.4659$ & $1.243\cdot 10^{-5}$ & $5.040\cdot 10^{-6}$ & $1.5555$ & $3.388\cdot 10^{-5}$ & $2.178\cdot 10^{-5}$\\
    $7$ & $2.4659$ & $3.107\cdot 10^{-6}$ & $1.260\cdot 10^{-6}$ & $1.5555$ & $8.471\cdot 10^{-6}$ & $5.446\cdot 10^{-6}$\\
    $8$ & $2.4659$ & $7.768\cdot 10^{-7}$ & $3.150\cdot 10^{-7}$ & $1.5555$ & $2.118\cdot 10^{-6}$ & $1.362\cdot 10^{-6}$\\
    $9$ & $2.4650$ & $8.577\cdot 10^{-4}$ & $3.478\cdot 10^{-4}$ & $1.5532$ & $2.331\cdot 10^{-3}$ & $1.499\cdot 10^{-3}$\\
   $10$ & $2.4655$ & $4.286\cdot 10^{-4}$ & $1.738\cdot 10^{-4}$ & $1.5543$ & $1.167\cdot 10^{-3}$ & $7.500\cdot 10^{-4}$\\
   $11$ & $2.4631$ & $2.833\cdot 10^{-3}$ & $1.149\cdot 10^{-3}$ & $1.5543$ & $1.165\cdot 10^{-3}$ & $7.489\cdot 10^{-4}$\\
\hline
 \end{tabular}
\end{table}
Note that the method produces already fairly accurate approximations in almost negligible time for $n=5$ or $n=6$. The over-proportional 
increase in computing time for higher discretization levels ($n=10$ and $n=11$) is due to longer execution time of the linear program solver,
and might indicate that the approximate problem is becoming ill-conditioned.
For $n=12$ and $m=12$, no reliable solution could be produced. In this case, the linear programming solver could find
no point satisfying the constraints, which can be circumvented by increasing the discretization level $m$ without increasing the number of
constraints $n$. However, this did not show better performance than the presented cases. The absolute error for $n=11$ is on a comparable level
to results obtained in \cite{rus-thesis}. Both the error of the cost criterion value and the $L^1$-error of the state space density are 
steadily decreasing, which is a strong indication of a convergent method, together with the presented convergence results.
The inferior approximation quality at $n=11$ compared to $n=10$ is believed to be due to the problem becoming ill-conditioned.\\
\begin{figure}[ht]
    \centering
    \begin{minipage}{0.46\textwidth}
        \centering
	\includegraphics[width=1\textwidth]{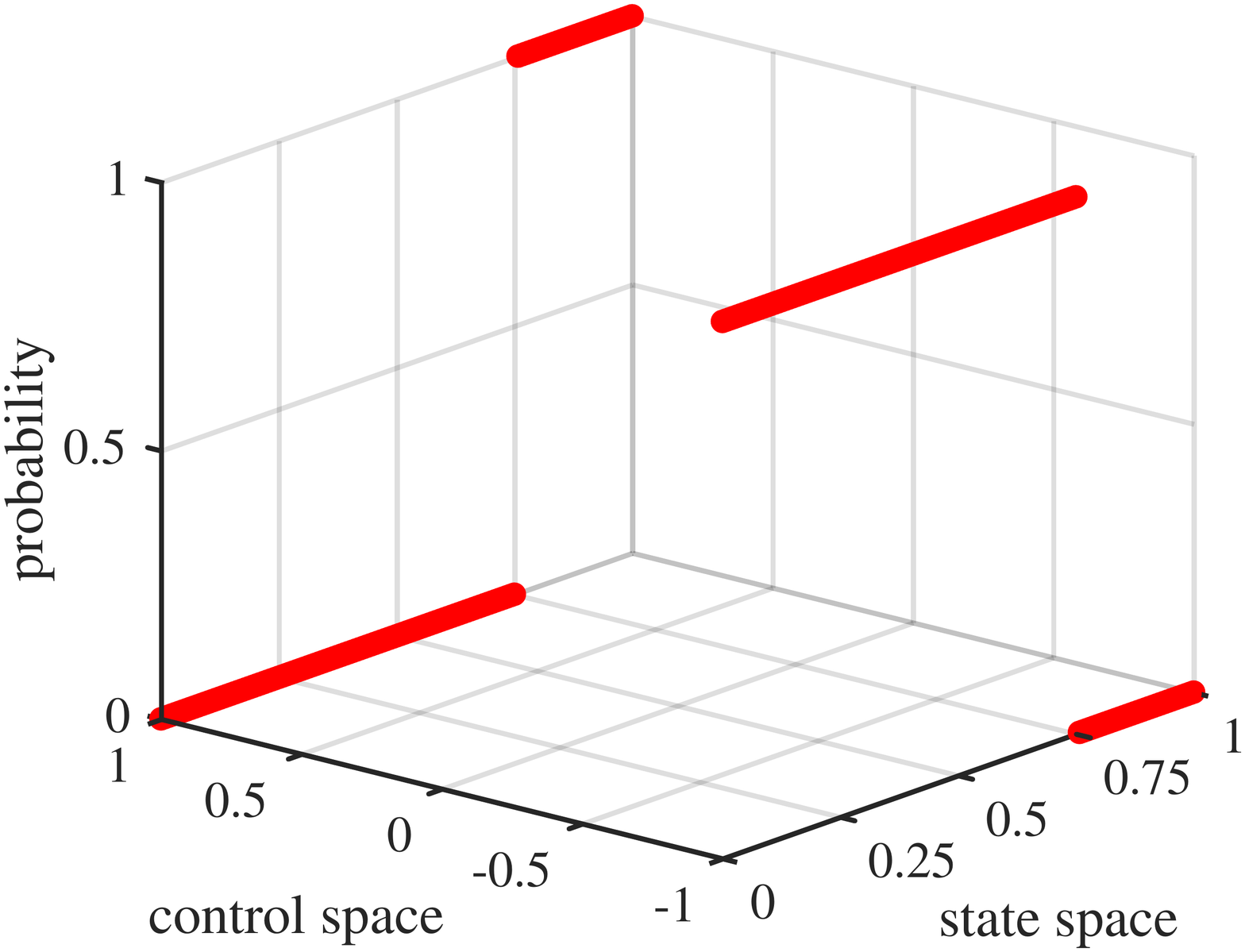}
        \caption{Computed optimal relaxed control, fine grid, modified bounded follower}
        \label{example:plot_mbf_relaxed_control_fine}
    \end{minipage}\hfill
    \begin{minipage}{0.46\textwidth}
        \centering
	\includegraphics[width=1\textwidth]{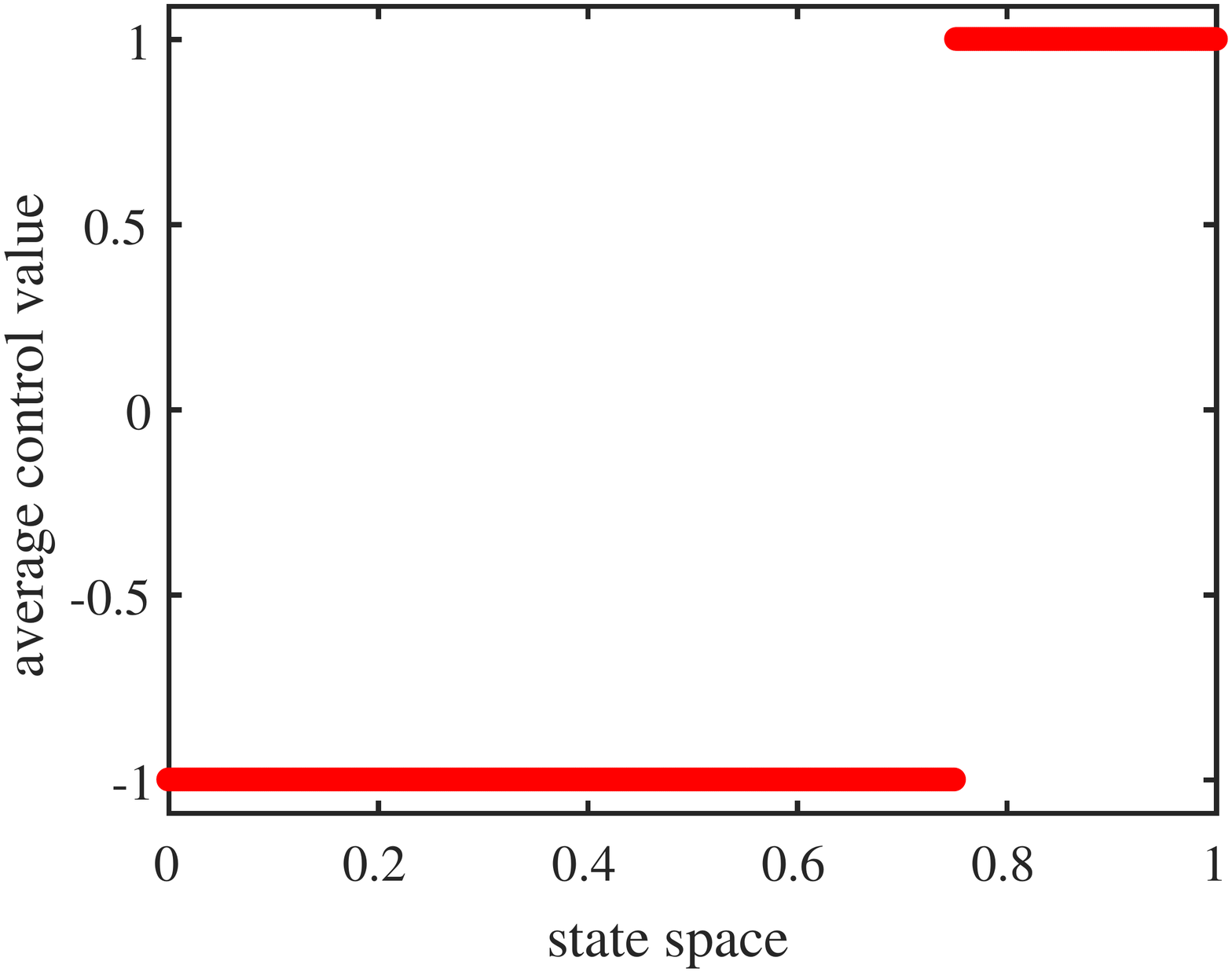}
        \caption{Computed optimal average control, fine grid, modified bounded follower}
        \label{example:plot_mbf_average_control_fine}
    \end{minipage}
\end{figure}
\noindent \Cref{example:plot_mbf_relaxed_control_coarse} shows the computed relaxed control $\hat{\eta}_0$ for $n=4$ and $m=4$. 
\Cref{example:plot_mbf_average_control_coarse} shows the average control value specified by this relaxed control. These figures have to be
understood as follows. \Cref{example:plot_mbf_relaxed_control_coarse} displays the full relaxed control, specifying the probability to pick
a certain control value $u$ when the process is in a certain state $x$. This state $x$ is found on the $x$-axis of the plot, labeled
`state space' and the choice of a control value $u$ corresponds to the $y$-axis of the plot, labeled `control space', while the probability
$\hat{\eta}_0(\{u\},x)$ of
picking this control value $u$ is presented on the $z$-axis, labeled `probability'. For example, at $x=0.25$, the control value $u=-1$ is chosen
with probability $1$, and the control value $u=1$ is chosen with probability $0$. We can see that for any possible value of $x$, $\hat{\eta}_0$
assigns full mass on either one of the two possible control values $u=-1$ and $u=1$. Hence, $\hat{\eta}_0$ can be represented by its average
control function, which is given by $x\mapsto \int_U u \hat{\eta}_0(du,x)$. It is shown in 
\Cref{example:plot_mbf_average_control_coarse}. In both \Cref{example:plot_mbf_relaxed_control_coarse} and
\Cref{example:plot_mbf_average_control_coarse}, the red dots represent the mesh points of the mesh $E^{(m)}$ as defined in 
(\ref{convergence:discrete-E-set-points}).
The switching point $a$ at $x=0.75$, where
the control switches from $-1$ to $+1$ is clearly visible in both figures.\\
The approximate state space density for $n=4$ and $m=4$, as displayed in blue in \Cref{example:plot_mbf_density_coarse}, clearly shows the
features inherited from the piecewise constant basis functions we use to approximate $p$. Its irregular pattern is due to the fact that we
introduced an additional mesh point in the middle of the state space. \Cref{example:plot_mbf_density_coarse} also shows the exact solution
displayed in red.\\
For a finer grid with parameters $n=10$ and $m=10$, \Cref{example:plot_mbf_relaxed_control_fine} shows the computed relaxed control 
$\hat{\eta}_0$. \Cref{example:plot_mbf_average_control_fine} shows the average control function.
The switching point $a$ again is clearly
visible. The red dots indicating the mesh points lie so dense that they form a solid line in both plots. 
\Cref{example:plot_mbf_density_fine} shows the approximate state space density for the parameter choice of $n=10$ and $m=10$. 
The exact solution could not be visually distinguished from the
approximate solution and is thus omitted from the figure. One can also see a change in concavity of the state space density at roughly $x=0.75$,
which is where the control switches its behavior from selecting $u=-1$ to $u=1$.
\begin{figure}[ht]
    \centering
    \begin{minipage}{0.46\textwidth}
        \centering
	\includegraphics[width=1\textwidth]{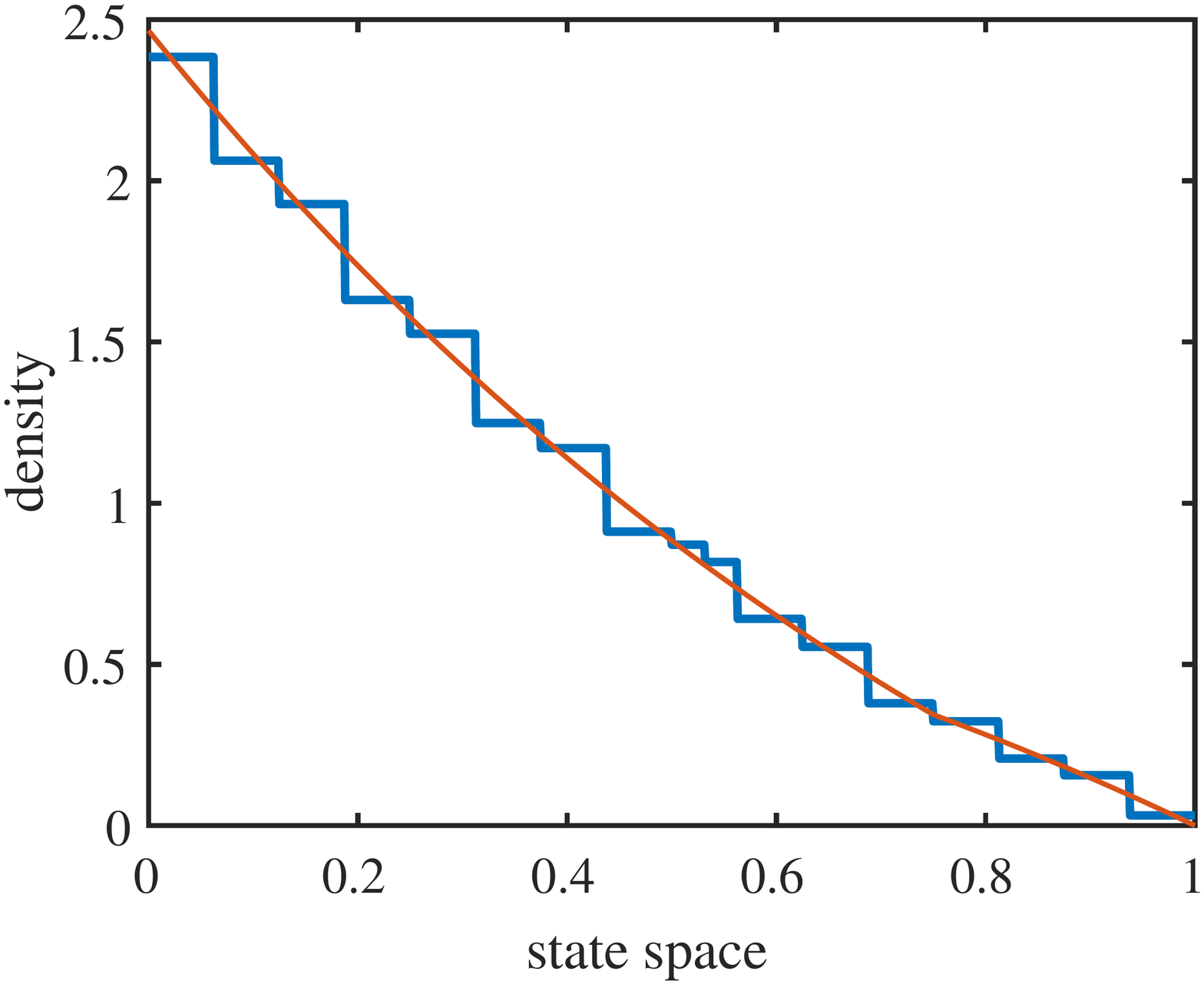}
        \caption{State space density, coarse grid, modified bounded follower}
        \label{example:plot_mbf_density_coarse}
    \end{minipage}\hfill
    \begin{minipage}{0.46\textwidth}
        \centering
	\includegraphics[width=1\textwidth]{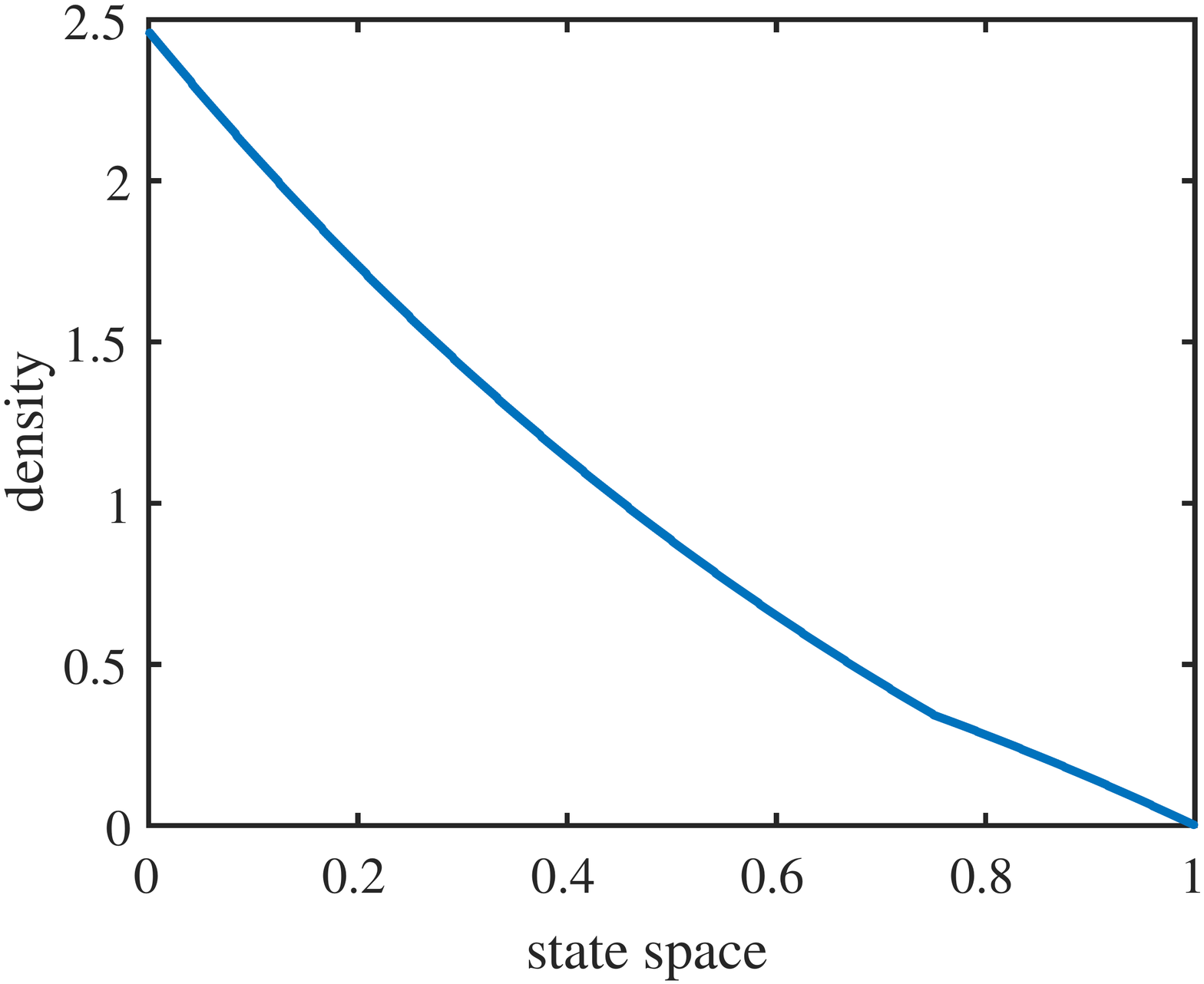}
        \caption{State space density, fine grid, modified bounded follower}
        \label{example:plot_mbf_density_fine}
    \end{minipage}
\end{figure}
\subsection{Simple Particle Problem with Costs of Control}
To illustrate the performance of the numerical method on a different type of problem, consider a stochastic control problem with state space 
$E=[-1,1]$ such that the process is governed by the SDE
\begin{equation*}
dX_t = u(X_t) dt +\sigma dW_t +d\xi_t,\quad X_0=x_0
\end{equation*}
in which $u(x)\in U=[-1,1]$. $\xi$ models reflections at both $-1$ (to the right) and $1$ (to the left), keeping the process inside of $E$. $X$ can be viewed as a particle
that randomly diffusions inside a confined space, and bounces off at the boundaries. Again, we adopt the long-term average cost criterion and
use the relaxed martingale formulation, compare Definition \ref{introduction:def-relaxed-mgp}. We retain the coefficient functions $b(x,u) = u$ and
$\sigma(x,u)\equiv \sigma$. To differentiate this example from the previous one, consider a cost structure given by $c_0 = x^2+u^2$ and
$c_1(x,u) \equiv c_1$ for some $c_1\in \mathbb{R}_+$ at both left endpoint $x=-1$ and right endpoint $x=1$. In particular, this means that using
the control induces a cost. In contrast to the modified bounded follower of \Cref{mbf-subsection}, we will see a different structure of the control since choosing the maximal or minimal control values
might not be optimal any longer, as this introduces additional costs. For this problem, no analytical solution is known to the authors.\\ 
We examine the influence of the cost of the reflection $c_1$ on the optimal
control. All subsequent calculations use $\sigma = \sqrt{2}/2$, $n=9$, $m=9$ and $k_m=m+3=12$. The latter is needed to attain a sufficient
approximation of the cost function, compare (\ref{convergence:discrete-U-bounds}).\\
\begin{figure}[bh]
    \centering
    \begin{minipage}{0.46\textwidth}
        \centering
	\includegraphics[width=1\textwidth]{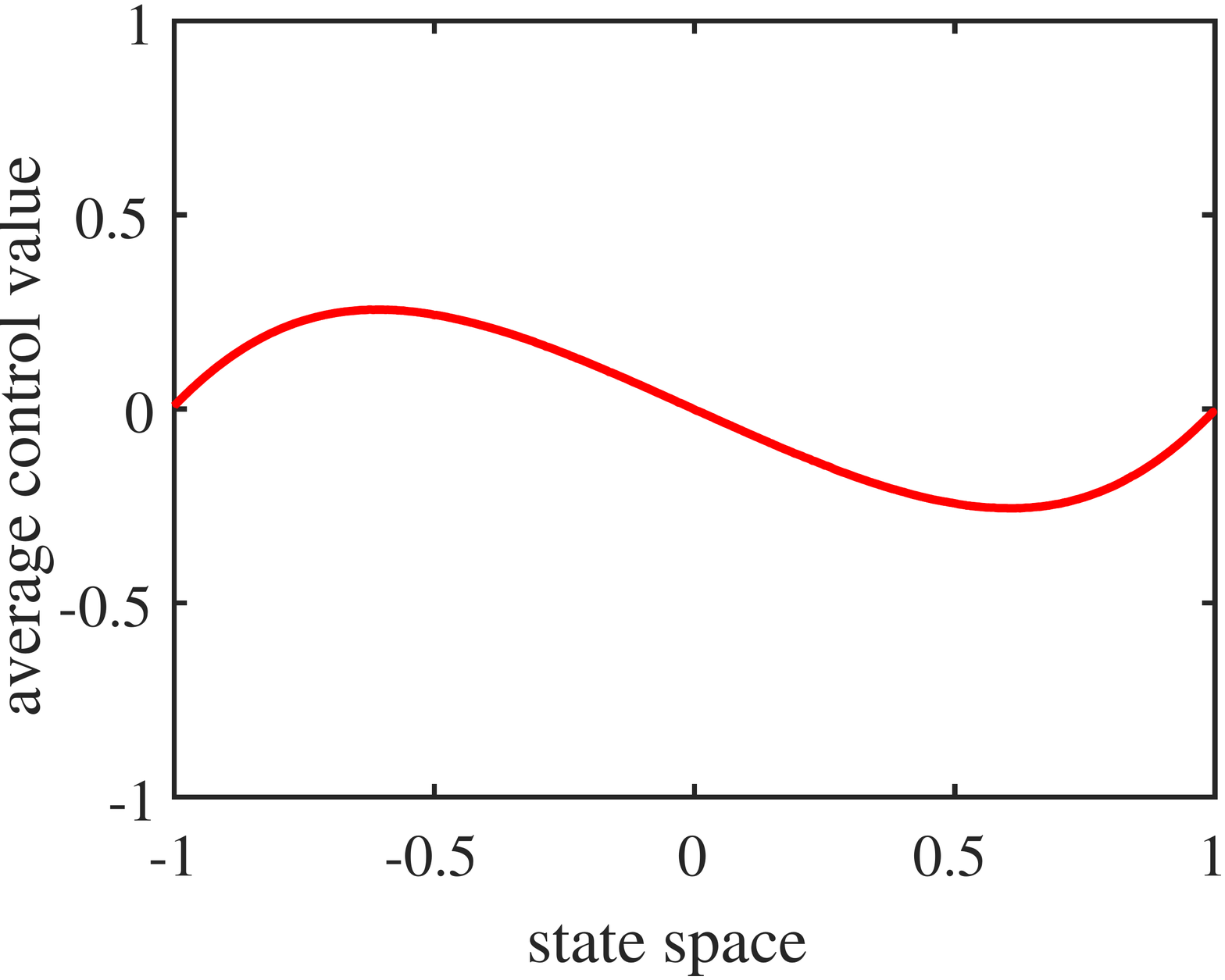}
        \caption{Average of optimal control, $c_1=0.01$, simple particle problem}
        \label{example-spp:plot_control-c_1-0.01}
    \end{minipage}\hfill
    \begin{minipage}{0.46\textwidth}
        \centering
	\includegraphics[width=1\textwidth]{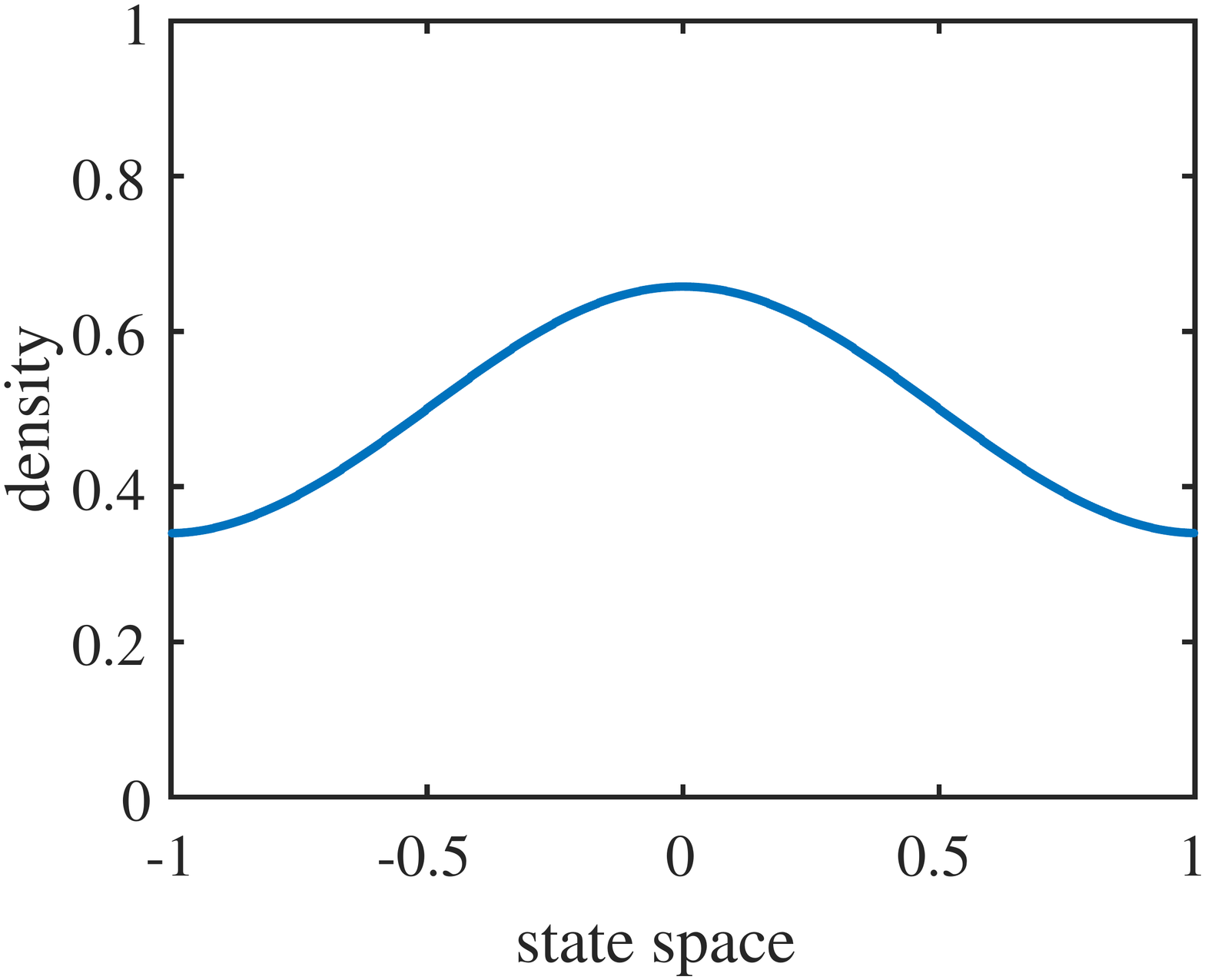}
        \caption{State space density, $c_1=0.01$, simple particle problem}
        \label{example-spp:plot_density-c_1-0.01}
    \end{minipage}
\end{figure}
\noindent \Cref{example-spp:plot_control-c_1-0.01} shows the average control $x\mapsto \int_U u\,\hat{\eta}_0(du,x)$ for a cost of reflections
given by $c_1=0.01$. We chose to show
a plot of the average control function rather than the full relaxed control since the numerical solutions
were degenerate relaxed controls, putting full mass on the values attained by $x\mapsto \int_U u\,\hat{\eta}_0(du,x)$,
with the exception of small 
rounding errors. Moreover, a full visualization of the relaxed control as seen in \Cref{mbf-subsection} is infeasible due to the high number of possible
control values. \Cref{example-spp:plot_density-c_1-0.01} shows the state space density associated with the control of 
\Cref{example-spp:plot_control-c_1-0.01}. The computed optimality criterion is $J^*_{n,m}=0.30259$.\\
\begin{figure}[h]
    \centering
    \begin{minipage}{0.46\textwidth}
        \centering
	\includegraphics[width=1\textwidth]{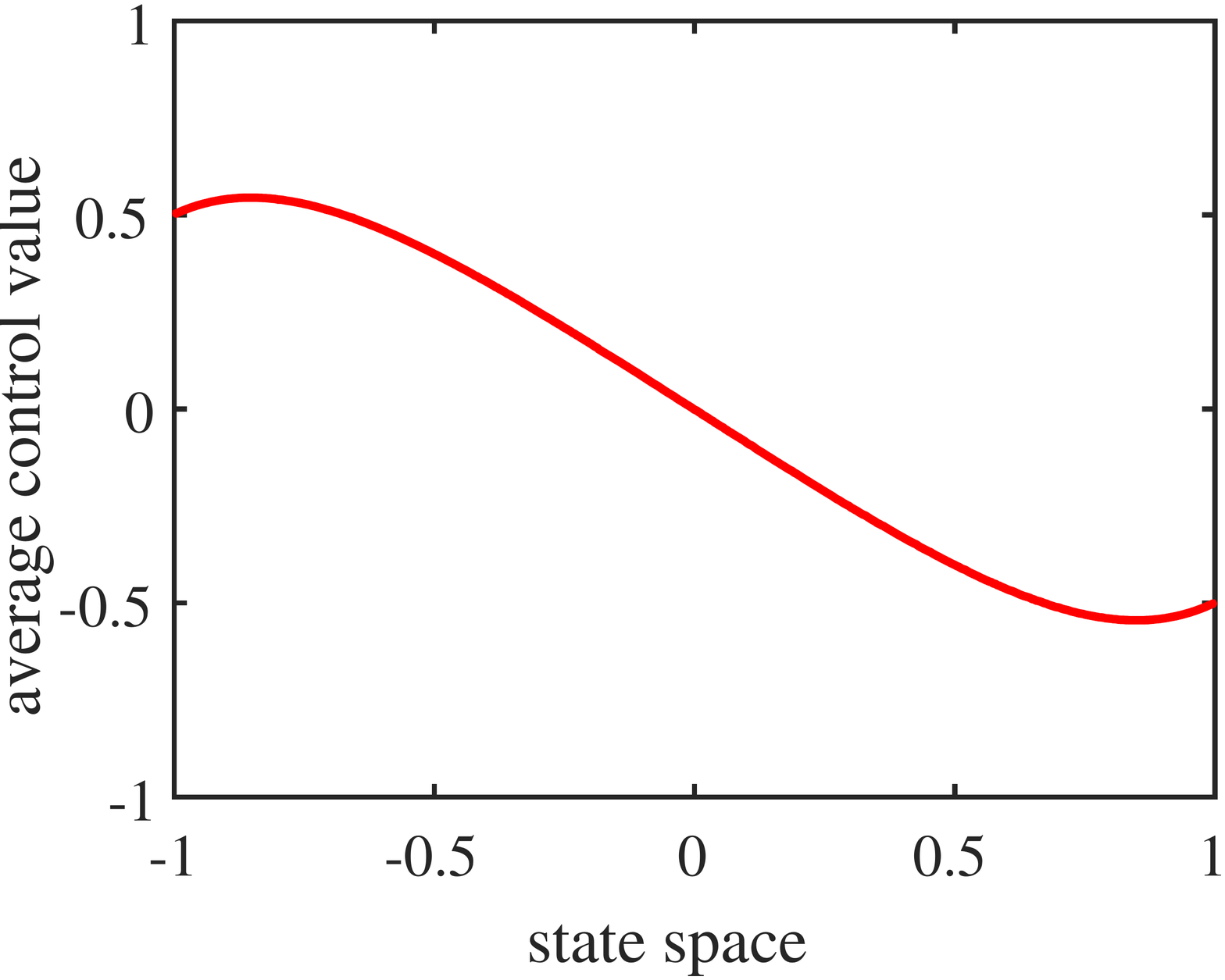}
        \caption{Average of optimal control, $c_1=1$, simple particle problem}
        \label{example-spp:plot_control-c_1-1}
    \end{minipage}\hfill
    \begin{minipage}{0.46\textwidth}
        \centering
	\includegraphics[width=1\textwidth]{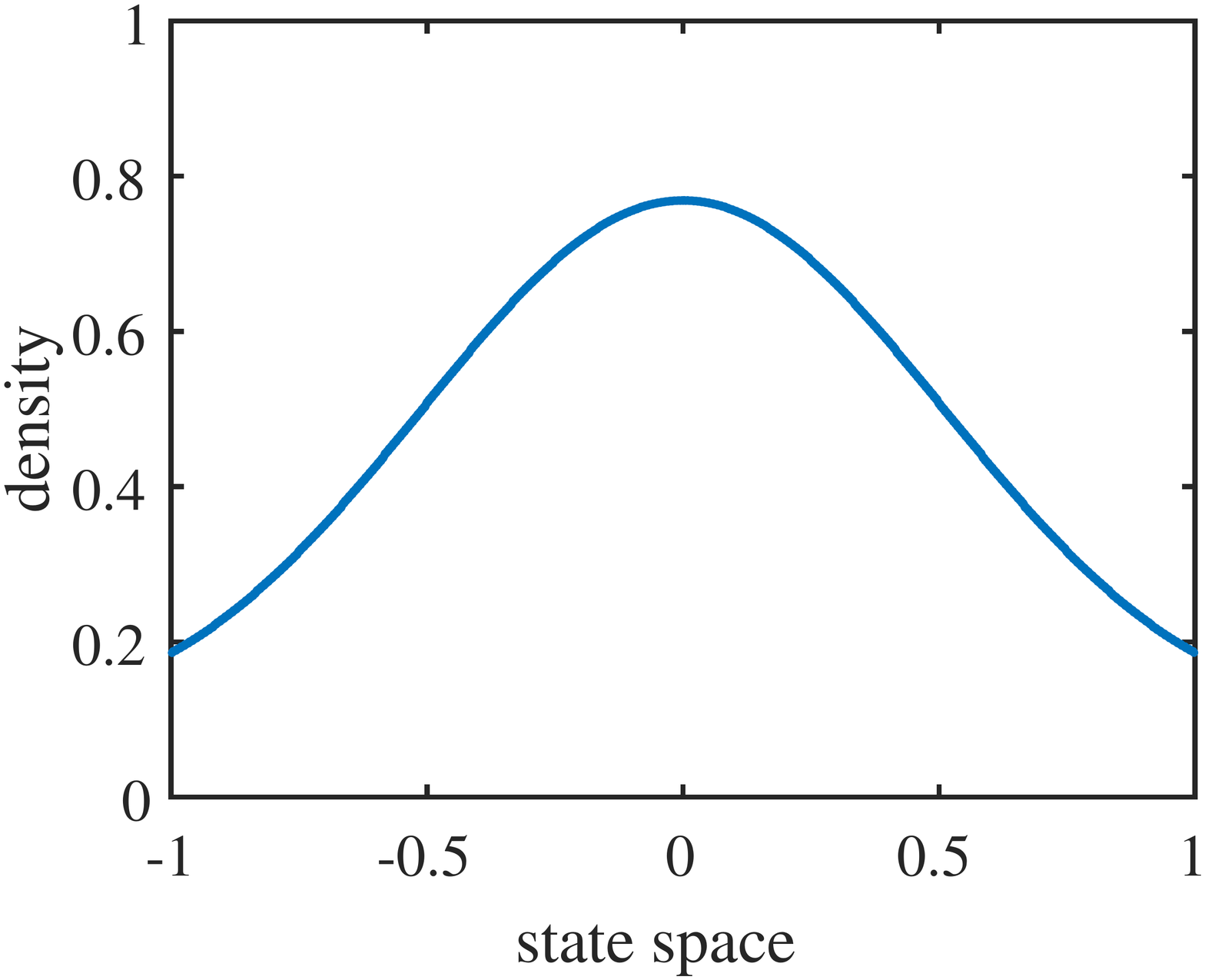}
        \caption{State space density, $c_1=1$, simple particle problem}
        \label{example-spp:plot_density-c_1-1}
    \end{minipage}
\end{figure}
\begin{figure}[h]
    \centering
    \begin{minipage}{0.46\textwidth}
        \centering
	\includegraphics[width=1\textwidth]{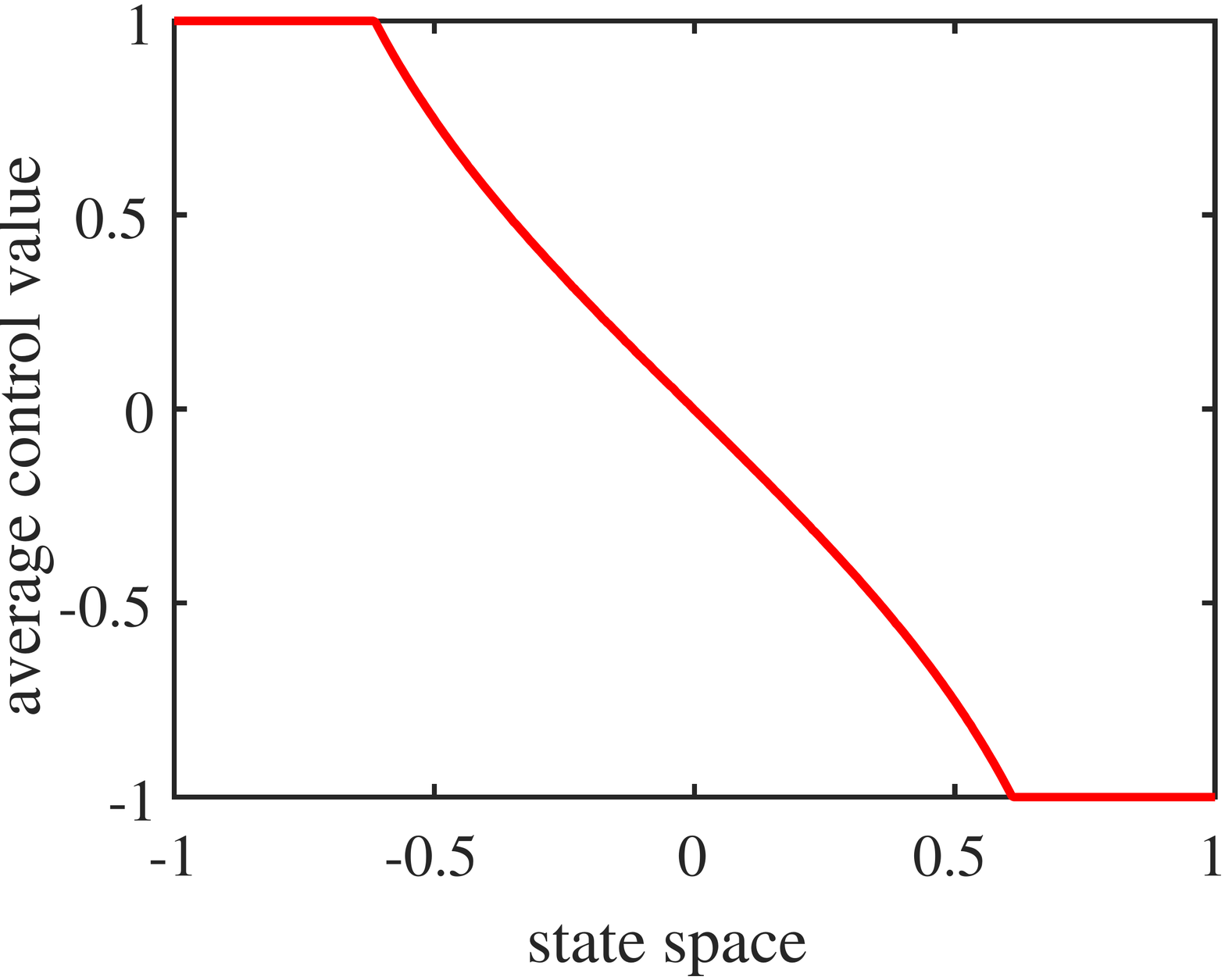}
        \caption{Average of optimal control, $c_1=6$, simple particle problem}
        \label{example-spp:plot_control-c_1-6}
    \end{minipage}\hfill
    \begin{minipage}{0.46\textwidth}
        \centering
	\includegraphics[width=1\textwidth]{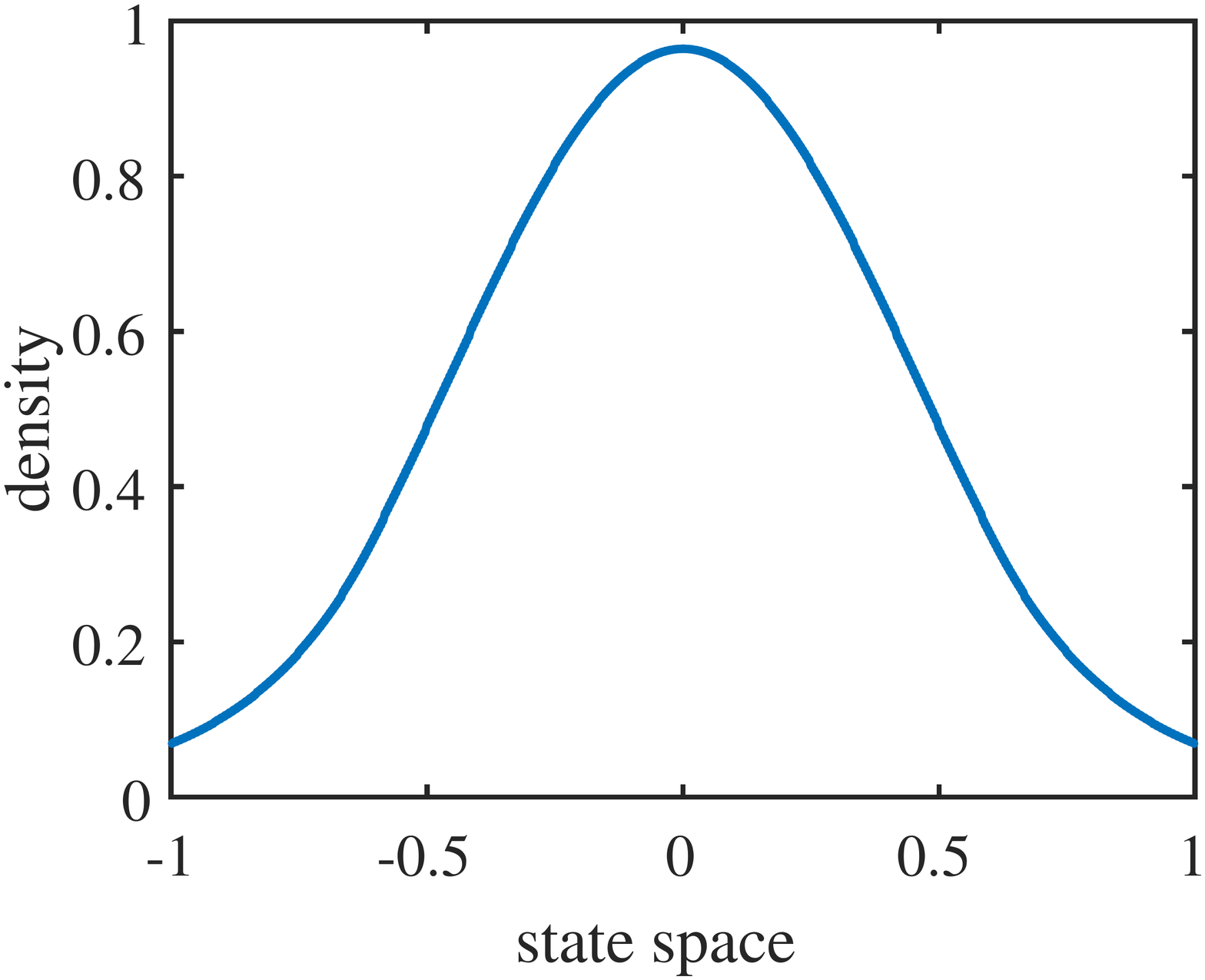}
        \caption{State space density, $c_1=6$, simple particle problem}
        \label{example-spp:plot_density-c_1-6}
    \end{minipage}
\end{figure}
\FloatBarrier
The interesting observation from this simulation is that the optimal control favors using the effect of the reflection to efficiently `push'
the process back into the interior. As the penalty for the reflection with $c_1 = 0.01$ is rather mild compared to the cost of 
using the control at full scale, increasingly less influence is enacted by the control as we move closer to both boundaries $-1$ and $1$.
\noindent \Cref{example-spp:plot_control-c_1-1} shows the optimal control when the costs of the reflection is increased to $c_1=1$. It reveals
that with a higher penalty for the reflection, it is beneficial to use the control more extensively, although a similar pattern as in the 
previous case can be observed when the process approaches the boundaries of the state space. The control is used slightly less in this area
to benefit from the reflection in direction of the origin.
The overall heavier use of the control results in a state space density (\Cref{example-spp:plot_density-c_1-1}) which is more concentrated
around the origin than the one from the previous case, see \Cref{example-spp:plot_density-c_1-0.01}. The value of the optimality criterion 
is given by $J^*_{n,m} = 0.42745$.\\
To illustrate an extreme case, we show a third example with $c_1=6$. \Cref{example-spp:plot_control-c_1-6} shows the optimal control in this setting,
\Cref{example-spp:plot_density-c_1-6} displays the state space density. In contrast to the previous two cases, the optimal control tries to avoid a
reflection under all circumstances by using its full force pushing back to the origin when the process approaches the boundaries
of the state space. Still, a trade-off is made when the process is close to the origin, and the control is used with less than full force to
avoid the costs induced by $c_0$. The state space density concentrates even more around the origin in this setting. 
The value of the cost criterion is given by $J^*_{n,m} = 0.66399$.

\section{Outlook}
The considerations presented in the present paper can be extended in several ways. From a numerical analysis point of view, it is highly
interesting how the numerical scheme behaves if higher-order basis elements are used to approximate the density $p$ of the state space
marginal of the occupation measure $\mu_0$. The analytic solution described in \Cref{mbf-subsection} has a density that is infinitely differentiable everywhere but at one point,
thus justifying the use of, for example, piecewise linear basis functions. However, this would require an adaption of the presented convergence
proof, in particular regarding the analysis leading up to the proof of Lemma \ref{analysis:inft_bss_bcs:bound_on_correction}. One aspect to be addressed is
the fact that as soon as we use standard elements with a order larger than $1$, for example, quadratic Lagrange elements,
the non-negativity of the approximate density cannot be guaranteed by restricting the coefficients to be non-negative.\\
Another topic to research would be the introduction of adaptive meshing techniques for both state and control space. Analytic or heuristic 
error estimator could guide a successive refinement of the meshes, leading to a increase in accuracy without significantly higher
computation time.\\
From a modeling point of view, on the one hand, an adaption of the discretization scheme for models featuring an unbounded state space would
enhance the number of applications for this numerical scheme. Several control problems in finance and economics feature an unbounded state space, and
are well suited for the linear programming approach. Initial investigations show that models with an unbounded state space can be 
approximated using a bounded state space with reflection boundaries. A full analysis of this approach would allow us to use the methods presented
in this paper in order to solve such models.\\
On the other hand, problems with finite time horizon or even optimal stopping problems could also
be solved with similar numerical techniques. While the analytic linear programming approach to address such problems is well studied, the discretization 
techniques presented in this chapter would have to be enhanced to reflect the time dependency of both constraint functions and measures. 
A numerical analysis of such techniques was conducted in \cite{lutz-lookback}, but a convergence analysis remained unconsidered.

\appendix
\section{Additional proofs}
\label{appendix_a}
This appendix provides the proofs of Lemma \ref{analysis:inft_bss_bcs:suitable_approximation},
Lemma \ref{analysis:inft_bss_bcs:bound_on_const_error} and Lemma \ref{analysis:inft_bss_bcs:bound_on_correction}.
\begin{proof}[Proof of Lemma \ref{analysis:inft_bss_bcs:suitable_approximation}]
 Find $\hat{\epsilon}_1<\epsilon$ such that $\lambda\left(\left\{x : p(x) \leq \hat{\epsilon}_1\right\}\right)<\frac{1}{2D_1}$, which is
 possible due to the continuity from above of measures. Define
 \begin{equation*}
 \bar{p}(x) = \begin{cases} 
      p(x) & p(x)>\hat{\epsilon}_1\\
      \hat{\epsilon}_1 & p(x)\leq \hat{\epsilon}_1 \\
   \end{cases}.
 \end{equation*}
 Then, 
 $\|p-\bar{p}\|_{L^1(E)}\leq \hat{\epsilon}_1\cdot \lambda\left(\left\{x : p(x) \leq \hat{\epsilon}_1\right\}\right)\leq \frac{\hat{\epsilon}_1}{2D_1}$.
 Now, choose $m_0$ large enough such that for all $m\geq m_0$, there is a 
 $\tilde{p}_{m}\in \mbox{span}\left(p_0,p_1,\ldots, p_{2^m-1}\right)$ with 
 $\|\bar{p}-\tilde{p}_{m}\|_{L^1(E)}\leq \frac{\hat{\epsilon}_1}{2D_1}$ and $\tilde{p}_{m}\geq\hat{\epsilon}_1$ (note there is no point
 in choosing $\tilde{p}_{m}<\hat{\epsilon}_1$ when approximating $\bar{p}$).
 Then,
 \begin{equation*}
 \|p-\tilde{p}_{m}\|_{L^1(E)} \leq \|p-\bar{p}\|_{L^1(E)} + \|\bar{p}-\tilde{p}_{m}\|_{L^1(E)}< \frac{\hat{\epsilon}_1}{2D_1} + \frac{\hat{\epsilon}_1}{2D_1}
 =\frac{\hat{\epsilon}_1}{D_1}
 \end{equation*}
 holds.
\end{proof}

\begin{proof}[Proof of Lemma \ref{analysis:inft_bss_bcs:bound_on_const_error}]
Fix $\delta>0$. 
 Since $(\mu_0,\mu_1)\in\mathscr{M}_{n,\infty}$, we have that for each $k=1,2,\ldots,n$
 \begin{equation*}
 Rf_k = \int_E\int_U Af_k(x,u)\eta_0(du,x)p(x)dx + \int_E \int_U Bf_k(x,u) \mu_1(dx\times du)
 \end{equation*}
 and thereby for any $\tilde{p}_m$ in the span of $\{p_0,p_1,\ldots,p_{2^m-1}\}$
 \begin{align*}
 &d^{(m)}_k(\tilde{p}_{m})\\
 &\: = Rf_k -\int_E \int_U Af_k(x,u) \hat{\eta}_{0,m}(du,x) \tilde{p}_{m}(x) dx - \int_E \int_U Bf_k(x,u) \hat{\eta}_{1,m}(du,x)\mu_{1,E}(dx)\\
 &\: = \int_E \int_U Af_k(x,u) \eta_0(du,x) p(x) dx-\int_E \int_U Af_k(x,u) \hat{\eta}_{0,m}(du,x) \tilde{p}_{m}(x) dx\\
 &\:\quad +\int_E \int_U Bf_k(x,u) \eta_1(du,x)\mu_{1,E}(dx)- \int_E \int_U Bf_k(x,u) \hat{\eta}_{1,m}(du,x)\mu_{1,E}(dx)
 \end{align*}
 holds.
 The triangle inequality reveals that 
 \begin{align*}
 &\vert d^{(m)}_k(\tilde{p}_{m}) \vert\\
 &\quad\leq \left\vert\int_E \int_U Af_k(x,u) \eta_0(du,x) p(x) dx -\int_E \int_U Af_k(x,u) \hat{\eta}_{0,m}(du,x) p(x) dx \right\vert \\
 &\quad+ \left\vert\int_E \int_U Af_k(x,u) \hat{\eta}_{0,m}(du,x) p(x) dx  -
 \int_E \int_U Af_k(x,u) \hat{\eta}_{0,m}(du,x) \tilde{p}_{m}(x) dx\right\vert\\
 &\quad +\left\vert\int_E \int_U Bf_k(x,u) \eta_1(du,x)\mu_{1,E}(dx)- \int_E \int_U Bf_k(x,u) \hat{\eta}_{1,m}(du,x)\mu_{1,E}(dx)\right\vert\\
 &\equiv \left\vert d^{(m)}_{k,1}\right\vert + \left\vert d^{(m)}_{k,2}\right\vert + \left\vert d^{(m)}_{k,3}\right\vert.
 \end{align*}
 Apply Lemma \ref{analysis:inft_bss_bcs:suitable_approximation} with $\epsilon = \delta$ and $D_1 = D_2\cdot 3 \cdot \max\{\bar{A},1\}$. Take 
 $\hat{\epsilon}_1$ and $m_1$ from this result. Set $\hat{\epsilon}_2 = \hat{\epsilon}_1/ D_2$. Then,
 $\hat{\epsilon}_2\leq\delta$ and for all $m\geq m_1$, there is a 
 $\tilde{p}_{m}\in \Span\{p_0,p_1,\ldots,p_{2^{m}-1} \}$ such that 
 $\|p-\tilde{p}_{m} \|_{L^1(E)}< \frac{\hat{\epsilon}_2}{3\cdot\rule{0pt}{9pt} \max\{\bar{A},1\}}$ as well as
 $\tilde{p}_{m}\geq D_2 \cdot  \hat{\epsilon}_2$ holds.
 Also,
 \begin{equation*}
 \left\vert d^{(m)}_{k,2}\right\vert \equiv \left\vert\int_E \int_U Af_k(x,u) \hat{\eta}_{0,m}(du,x) \left(p(x)-\tilde{p}_{m}(x)\right)dx\right\vert
 \leq \bar{A} \|p-\tilde{p}_{m}\|_{L^1(E)} < \frac{\hat{\epsilon}_2}{3}.
 \end{equation*}
 By Proposition \ref{convergence:cauchytrick_proposition2}, we can choose $m_2\geq m_1$ such that for all $m\geq m_2$, $\vert d^m_{k,1} \vert$ is bounded
 by $\frac{\hat{\epsilon}_2}{3}$. By Proposition \ref{convergence:cauchytrick_proposition2_singular}, we can choose $m_3\geq m_2$ such that 
 $\vert d^{(m)}_{k,3} \vert$ is bounded by $\frac{\hat{\epsilon}_2}{3}$ for all $m\geq m_3$, which shows that 
 $\left\vert d_k^{(m)}(\tilde{p}_{m})\right\vert <\hat{\epsilon}_2$ for 
 $k=1,2,\ldots,n$. For $k=n+1$, since $p$ is a probability density,
 \begin{equation*}
 \|\tilde{p}_{m}\|_{L^1(E)} \leq\|\tilde{p}_{m}-p\|_{L^1(E)} + \|p\|_{L^1(E)}<\frac{\hat{\epsilon}_2}{3\max\{\rule{0pt}{9pt}\bar{A},1\}} + 1< \hat{\epsilon}_2+1.
 \end{equation*}
 Now assume that $\|\tilde{p}_{m}\|_{L^1(E)}<1-\hat{\epsilon}_2$. Then,
 \begin{equation*}
 \|p\|_{L^1(E)} \leq\|p-\tilde{p}_{m}\|_{L^1(E)} + \|\tilde{p}_{m}\|_{L^1(E)}<\frac{\hat{\epsilon}_2}{3\max\{\rule{0pt}{9pt}\bar{A},1\}}+1-\hat{\epsilon}_2
 < \hat{\epsilon}_2 + 1-\hat{\epsilon}_2 =1,
 \end{equation*}
 a contradiction, and we have that
 \begin{equation*}
 1-\hat{\epsilon}_2 \leq \|\tilde{p}_{m}\|_{L^1(E)}\leq 1+\hat{\epsilon}_2.
 \end{equation*}
 Hence, $\vert d^{(m)}_{n+1}(\tilde{p}_{m}) \vert < \hat{\epsilon}_2$, which completes the proof, upon setting $m_0=m_3$.
\end{proof}
Now we can show that the statement of Lemma \ref{analysis:inft_bss_bcs:bound_on_correction} is true.
\begin{proof}[Proof of Lemma \ref{analysis:inft_bss_bcs:bound_on_correction}]
 Fix $\vartheta>0$. Select $m_1\in\mathbb{N}$ large enough such that for all $m\geq m_1$, $C^{(m)}$ has full rank and thus 
 $n+1$ independent columns. For any $m\geq m_1$, let 
 $\bar{C}^{(m)}\in\mathbb{R}^{n+1,n+1}$ be a matrix consisting of $n+1$ independent columns of $C^{(m)}$. Set 
 \begin{equation*}
 \delta = \frac{\vartheta}{\rule{0pt}{10pt}\max\left\{1,\left\|\left(\bar{C}^{(m_1)}\right)^{-1}\right\|_\infty\right\}}
 \end{equation*}
 and by Lemma \ref{analysis:inft_bss_bcs:bound_on_const_error},
 with $\delta$ and $D_2=\max\{1,\|\left(\bar{C}^{(m_1)}\right)^{-1}\|_\infty\}$, find $m_2\geq m_1$ such that for
 all $m\geq m_2$, there is a $\tilde{p}_{m}$, with $\|d^{(m)}(\tilde{p}_{m})\|_\infty<\hat{\epsilon}_2\leq\delta$, for some $\hat{\epsilon}_2>0$,
 satisfying 
 \begin{equation*}
\tilde{p}_{m}\geq\max\left\{1,\|\left(\bar{C}^{(m_1)}\right)^{-1}\|_\infty\right\}\cdot \hat{\epsilon}_2\geq\|\left(\bar{C}^{(m_1)}\right)^{-1}\|_\infty\cdot \hat{\epsilon}_2 
 \end{equation*}
 as well as
 $\|p-\tilde{p}_{m}\|_{L^1(E)}<\frac{\hat{\epsilon}_2}{3\cdot\rule{0pt}{9pt} \max\{\bar{A},1\}}$. Set 
 $\hat{\vartheta} = \max\left\{1,\|\left(\bar{C}^{(m_1)}\right)^{-1}\|_\infty\right\}\cdot \hat{\epsilon}_2$ and note that 
 $\frac{\hat{\epsilon}_2}{3\cdot\rule{0pt}{9pt} \max\{\bar{A},1\}}<\hat{\vartheta}$.
 Consider the solution $\tilde{y}\in \mathbb{R}^{2^{m_1}}$ for 
 $C^{(m_1)}y=-d^{(m_2)}(\tilde{p}_{m})$ that is given by injecting
 $\bar{y} = \left(\bar{C}^{(m_1)}\right)^{-1}\left(-d^{(m_2)}(\tilde{p}_{m})\right)\in\mathbb{R}^{n+1}$ into $\mathbb{R}^{2^{m_1}}$. Then, 
 \begin{equation*}
 \|\tilde{y}\|_\infty = \|\bar{y}\|_\infty = \left\|\left(\bar{C}^{(m_1)}\right)^{-1}d^{(m_2)}(\tilde{p}_{m}) \right\|_\infty
 < \left\|\left(\bar{C}^{(m_1)}\right)^{-1}\right\|_\infty\|d^{(m_2)}(\tilde{p}_{m})\|_\infty\leq\hat{\vartheta}.
 \end{equation*}
 We now show that there is a solution $\tilde{y}$ to $C^{(m_2)}y= d^{(m_2)}(\tilde{p}_{m})$ that satisfies $\|\tilde{y}\|_\infty\leq\hat{\vartheta}$.
 By the definition of the constraint matrix, for any $m\in \mathbb{N}$, we have that for $k=1,2,\ldots,n+1$ and $i=0,1,\ldots, 2^{m}-1$,
 \begin{equation*}
 C^{(m+1)}_{k,2i} + C^{(m+1)}_{k,2i+1} = C^{(m)}_{k,i}
 \end{equation*}
 holds. Indeed, since for $1\leq k \leq n$, by the choice of basis functions $\{p_0,p_1,\ldots,p_{2^{m}-1}\}$ as indicator functions
 over dyadic intervals, the entries of $C^{(m)}_{k,i}$ are given by
 integration of the functions $Af_k$ over intervals that are cut in half, and if $k=n+1$, the entries are simply given by the
 interval lengths $(x_{j+1} -x_j)$ since $p_j=1$ on $[x_{j+1},x_j)$. Hence, if $y$ is a solution to $C^{(m)} y=-d$, 
 the vector $\bar{y}\in \mathbb{R}^{2^{m+1}}$ with
 components
 \begin{equation*}
 \bar{y}_{2i+1} = \bar{y}_{2i} = y_i
 \end{equation*}
 where $i=0,1,\ldots, 2^{m}-1$, satisfies $C^{(m+1)}y=-d$, and $\|y\|_\infty = \|\bar{y}\|_\infty$ holds. Inductively, this reveals that
 for any $m\geq m_1$, there is a solution $\tilde{y}$ to $C^{(m)}y = -d^{(m_2)}(\tilde{p}_{m})$ which satisfies
 $\|\tilde{y}\|_\infty = \|y\|_\infty \leq \hat{\vartheta}$. In particular, this means that there is a solution
 $\tilde{y}$ to $C^{(m_2)}y=-d^{(m_2)}(\tilde{p}_{m})$, with $\|\tilde{y}\|_\infty = \|y\|_\infty < \hat{\vartheta}$. For any $m\geq m_2$, this analysis can be conducted
 similarly, showing the result for $m_0=m_2$.
\end{proof}
\bibliographystyle{siam}
\bibliography{dissertation}
\end{document}